\documentclass[final,onecolumn]{IEEEtran}

\usepackage{amsmath, amssymb, bbm, xspace}

\newcommand{\Real}{\ensuremath{\mathbb{R}}}

\def\spose#1{\hbox to 0pt{#1\hss}}

\def\text #1{\hbox{\quad#1\quad}}

\def\x{{\mathbf{x}}}
\def\w{{\mathbf{w}}}
\def\z{{\mathbf{z}}}
\def\u{{\mathbf{u}}}

\def\nthinsp{\mskip -2   mu}

\def\superstar{^{\raise 0.5pt\hbox{$\nthinsp *$}}}
\def\SUPERSTAR{^{\raise 0.5pt\hbox{$*$}}}

\def\lamstarT {\lambda^{\raise 0.5pt\hbox{$\nthinsp *$}T}}

\def\Uscr{{\cal U}}

\def\Xscr{{\cal X}}

\def\hbar{\skew{4.2}\bar h}

		\def\bkE{{\rm I\kern-.17em E}}
		\def\bk1{{\rm 1\kern-.17em l}}
		\def\bkD{{\rm I\kern-.17em D}}
		\def\bkR{{\rm I\kern-.17em R}}
		\def\bkR{\mathbb{R}}
		\def\bkP{{\rm I\kern-.17em P}}
		\def\bkY{{\bf \kern-.17em Y}}
		\def\bkZ{{\bf \kern-.17em Z}}

		\def\beq{\begin{eqnarray}}
		\def\bc{\begin{center}}
		\def\be{\begin{enumerate}}
		\def\bi{\begin{itemize}}
		\def\bs{\begin{small}}
		\def\bS{\begin{slide}}
		\def\ec{\end{center}}
		\def\ee{\end{enumerate}}
		\def\ei{\end{itemize}}
		\def\es{\end{small}}
		\def\eS{\end{slide}}
		\def\eeq{\end{eqnarray}}

	\def\cp2problem#1#2#3#4{\fbox
		 {\begin{tabular*}{0.9\textwidth}
			{@{}l@{\extracolsep{\fill}}l@{\extracolsep{6pt}}l@{\extracolsep{\fill}}c@{}}
				#1 & & $#4 $ 
			\end{tabular*}}}
\newcommand{\pmat}[1]{\begin{pmatrix} #1 \end{pmatrix}}
		
		\renewcommand{\emph}[1]{\textbf{#1}}

		\def\bkE{{\rm I\kern-.17em E}}
		\def\bk1{{\rm 1\kern-.17em l}}
		\def\bkD{{\rm I\kern-.17em D}}
		\def\bkP{{\rm I\kern-.17em P}}
		
		\def\bkZ{{\bf{Z}}}

\newcommand {\beeq}[1]{\begin{equation}\label{#1}}
\newcommand {\eeeq}{\end{equation}}
\newcommand {\bea}{\begin{eqnarray}}
\newcommand {\eea}{\end{eqnarray}}

\def\texitem#1{\par\smallskip\noindent\hangindent 25pt
               \hbox to 25pt {\hss #1 ~}\ignorespaces}

\usepackage{cite,color}
\usepackage{amsthm,amsmath,amssymb,amsfonts,verbatim}
\usepackage{algorithmic}
\usepackage{graphicx}
\usepackage{textcomp}
\def\bm#1{\ensuremath \textcolor{black}{{{\boldsymbol{#1}}}}}
\usepackage{dsfont,threeparttable,mathtools,multirow}

\DeclareMathOperator*{\argmax}{argmax}

\usepackage{array}
\newcommand{\PreserveBackslash}[1]{\let\temp=\\#1\let\\=\temp}
\newcolumntype{C}[1]{>{\PreserveBackslash\centering}p{#1}}
\newcolumntype{R}[1]{>{\PreserveBackslash\raggedleft}p{#1}}
\newcolumntype{L}[1]{>{\PreserveBackslash\raggedright}p{#1}}

\newtheorem{theorem}{Theorem}
\newtheorem{lemma}{Lemma}
\newtheorem{definition}{Definition}

\newtheorem{proposition}{Proposition}
\newtheorem{corollary}{Corollary}
\newtheorem{assumption}{Assumption}

\newcommand{\tabincell}[2]{\begin{tabular}{@{}#1@{}}#2\end{tabular}}
\newcommand\mtiny[1]{\mbox{\tiny\ensuremath{#1}}}
\newcommand{\ic}[1]{{\color{black}#1}}
\newcommand{\uss}[1]{{\color{black}#1}}
\allowdisplaybreaks

\begin{document}
\title{Variance-Reduced Splitting Schemes for Monotone Stochastic Generalized Equations}
\author{Shisheng Cui  and Uday V. Shanbhag
\thanks{(\ic{To be modified})This paragraph of the first footnote will contain the date on 
which you submitted your paper for review. It will also contain support 
information, including sponsor and financial support acknowledgment. For 
example, ``This work was supported in part by the U.S. Department of 
Commerce under Grant BS123456.'' }
\thanks{Shisheng Cui and Uday V. Shanbhag are with the Department of Industrial and Manufacturing Engineering, The Pennsylvania State University, University Park, PA 16802 USA (e-mail: suc256@psu.edu; udaybag@psu.edu). }}

\maketitle

\begin{abstract}
We consider monotone inclusion problems where the operators may be
expectation-valued, a class of problems that subsumes convex stochastic
optimization problems as well as subclasses of stochastic variational
inequality and equilibrium problems. A direct application of splitting
schemes is complicated by the need to resolve problems with
expectation-valued maps at each step, a concern that is addressed by
using sampling. Accordingly, we propose an avenue for addressing
uncertainty in the mapping: {\em Variance-reduced stochastic modified
forward-backward splitting scheme} ({\bf vr-SMFBS}). In constrained
settings, we consider structured settings when the map can be
decomposed into an expectation-valued map $A$ and a maximal monotone
map $B$ with a tractable resolvent.  We show that the proposed schemes
are equipped with a.s. convergence guarantees, linear (strongly
monotone $A$) and $\mathcal{O}(1/k)$ (monotone $A$) rates of
convergence while achieving optimal oracle complexity bounds. 
The rate statements in monotone regimes appear to be amongst the first
and rely on leveraging the Fitzpatrick gap function for monotone
inclusions. Furthermore, the schemes rely on weaker moment requirements
on noise and allow for weakening unbiasedness requirements on
oracles in strongly monotone regimes. Preliminary numerics on a class
of two-stage stochastic variational inequality problems reflect these
findings and show that the variance-reduced schemes outperform
stochastic approximation schemes and sample-average approximation
approaches. The benefits of attaining deterministic
rates of convergence become even more salient when resolvent
computation is expensive.
\end{abstract}


\section{Introduction}
\label{sec:introduction}
\IEEEPARstart{T}{he} generalized equation (\uss{alternately referred to as the inclusion problem}) represents a crucial
mathematical object in decision and control theory, representing
a set-valued generalization to the more standard root-finding problem
which requires solving $F(x) = 0$, where  $F: \mathbb{R}^n \to \mathbb{R}^n$ \uss{is single-valued}.  Specifically, if $T$ is
a set-valued map, defined as $T: \mathbb{R}^n \rightrightarrows
\mathbb{R}^n$, and $T$ is characterized by a distinct structure in that
it can be cast as the sum of two operators $A$ and $B$, then the generalized equation (GE) takes the form 
\begin{align*}\tag{GE} 0 \in T(x) \triangleq A(x)+B(x).
\end{align*}
Here $A: \mathbb{R}^n \to \mathbb{R}^n$ is a single-valued map and $B: \mathbb{R}^n \rightrightarrows \mathbb{R}^n$ is a set-valued map.
While such objects have a storied history, an excellent overview was first
provided by Robinson~\cite{robinson83generalized}. Generalized equations have
been extensively examined since the 70s when
Rockafellar~\cite{rockafellar1976monotone} developed a proximal point scheme
for a generalized equation characterized by monotone operators. In fact, this
scheme subsumes a range of well known schemes such as the augmented Lagrangian
method~\cite{glowinski1989augmented}, Douglas-Rachford
splitting~\cite{douglas1956numerical}, amongst others. It can be observed that a large class of optimization and equilibrium problems can be modeled as \eqref{GE},
including the necessary conditions of nonlinear programming problems,
variational inequality and complementarity problems, and a broad range of
equilibrium problems (cf.~\cite{robinson83generalized}). 
Under
suitable requirements on $A$ and $B$, a range of splitting
methods can be developed and has represented a vibrant area of
research over the last two
decades~\cite{douglas1956numerical,peaceman1955numerical,lions1979splitting,passty1979ergodic}.

In this paper, we
consider addressing the {\em stochastic} counterpart of generalized equations,
a class of problems that has seen recent study via sample-average
approximation (SAA) techniques~\cite{chen19convergence}.  Formally, the
stochastic generalized equation requires an $x \in\mathbb{R}^n$
such that 
\begin{align} \tag{SGE}
0 \in \mathbb{E}[T(x,\xi(\omega))] \triangleq \mathbb{E}[A(x,\xi(\omega))] + B(x), 
\end{align}
where the components of the map $A$ are denoted by $A_i$, $i=1,\dots,n$, $\xi:\Omega \to \mathbb{R}^d$ is a random variable, $A_i:
\mathbb{R}^n \times \Omega \rightrightarrows \mathbb{R}^n$ is a set-valued map, $\mathbb{E}[\cdot]$ denotes the expectation, and the associated
probability space is given by
$(\Omega, {\cal F}, \mathbb{P})$. In the remainder of this paper, we refer to $A(x,\xi(\omega))$ by $A(x,\omega)$. The expectation of a set-valued map leverages the Aumann integral~\cite{aumann1965integrals} and is formally defined as
$\mathbb{E}[A_i(x,\xi(\omega))]= \left\{ \int v_i(\omega)dP(\omega)\mid \quad v_i(\omega)\in A_i(x,\xi(\omega)) \right\}.$
Consequently, the expectation $\mathbb{E}[A(x,\omega)]$ can be defined as a Cartesian product of the sets $\mathbb{E}[A_i(x,\omega)]$, defined as 
$\mathbb{E}[A(x,\omega)]  \ \triangleq \ \prod_{i=1}^n\mathbb{E}[A_i(x,\omega)].$ 
We motivate (SGE) by considering some examples.  Consider
the stochastic convex optimization
problem~\cite{dantzig2010linear,birge2011introduction,shapiro2014lectures} given by
${\displaystyle \min_{x \in \mathcal{X}}} \, \mathbb{E}[g(x,\omega)],$ where
$g(\bullet,\omega)$ is a differentiable convex function for every $\omega$
and $\mathcal{X}$ is a closed and convex set. Such a problem can be
equivalently stated as $0 \in T(x) \triangleq
\mathbb{E}[G(x,\omega)] + \mathcal{N}_\mathcal{X}(x)$, where $G(x,\omega) = \nabla g(x,\omega)$ and $\mathcal{N}_\mathcal{X}(x)$ denotes the normal cone of $\mathcal{X}$ at $x$. In
fact, the
single-valued
stochastic variational inequality problems~\cite{jiang08stochastic,juditsky2011solving,shanbhag2013stochastic} can be cast as stochastic
inclusions as well as seen by $0 \in T(x) \triangleq
\mathbb{E}[F(x,\omega)] + \mathcal{N}_\mathcal{X}(x)$, where
$F(\bullet,\omega)$ is a realization of the mapping. This introduces a
pathway for examining stochastic analogs of traffic
equilibrium~\cite{ravat17existence} and Nash equilibrium
problems~\cite{ravat11characterization} as well as a myriad of other
problems subsumed by variational inequality
problems~\cite{facchinei2007finite}. We describe two problems that have
seen recent study which can also be modeled as SGEs, allowing for
developing new computational techniques.

\subsection{Two motivating examples} \label{sec:app}
\noindent {\em (a) A subclass of stochastic multi-leader multi-follower games.} 
Consider a class of multi-leader
multi-follower games~\cite{sherali1984multiple,pang05quasi,su2007analysis,kulkarni15existence} with ${\bf N}$ 
leaders, denoted by $\{1, \cdots, {\bf N}\}$ and
$M$ followers, given by $\{1, \cdots, {\bf M}\}.$ In
general, this class of games is challenging to analyze since the player
problems are nonconvex and early existence statements have relied on
eliminating follower-level decisions, leading to a noncooperative game with convex nonsmooth
player problems. Adopting a similar approach in examining a stochastic
generalization of a quadratic setting examined
in~\cite{steffensen18quadratic} with a single follower where ${\bf M} = 1$, suppose the follower problem is 
\begin{align}\tag{\mbox{Follow$_i(x_i)$}}
	\min_{y_i \, \geq \, l_i(x_i)} & \quad \tfrac{1}{2} y_i^\mathsf{T} Q_i y_i - b_i(x_i)^\mathsf{T}  y_i,\end{align}
where $Q_i$ is a positive definite and diagonal matrix, $b_i(\bullet)$ and $l_i(\bullet)$ are affine functions.
Suppose the leaders compete in a Cournot game in which the $i$th leader  solves 
\begin{align}\tag{\mbox{Leader$_i(x_{-i})$}}
	\min_{x_i \, \in \, \mathcal{X}_i} & \quad c_i(x_i) - \mathbb{E}[p(X,\omega) x_i] + a_iy_i(x_i), 
\end{align}
where $c_i: \mathcal{X}_i \to \mathbb{R}$ is a smooth convex function, the
inverse-demand function $p(\bullet)$ is defined as $p(X)
\triangleq d(\omega)-r(\omega)X$, $d(\omega), r(\omega) > 0$ for every $\omega \in \Omega$, $X
\triangleq \sum_{i=1}^Nx_i$, $y_i(x_i)$ denotes a
best-response of follower $i$, and $\mathcal{X}_i$ is a
closed and convex set in $\mathbb{R}$. Follower $i$'s 
best-response $y_i(x)$, given leader-level decisions $x$, can be derived by considering the necessary
and sufficient conditions of optimality: \begin{align*}
	y_i(x_i) = \max\{ Q_i^{-1} b_i(x_i), l_i(x_i)\}.  
\end{align*}
Consequently, we may eliminate the follower-level decision in the leader level problem, leading to a nonsmooth stochastic Nash equilibrium problem given by the following:
\begin{align}
\tag{\mbox{Leader$_i(x_{-i})$}}
\begin{aligned}
	\min_{x_i \in \mathcal{X}_i} & c_i(x_i) - \mathbb{E}[p(X,\omega)x_i] \\
	& + {a_i^\mathsf{T}  \max\{Q_i^{-1}b_i(x_i), l_i(x_i)\}}. 
\end{aligned}
\end{align}
Under convexity of $b_i(\bullet)$ and $l_i(\bullet)$, and suitable assumptions on $Q_i$ and $a_i$, the expression $a_i^\mathsf{T}  \max\{Q_i^{-1}b_i(x_i), l_i(x_i)\}$ is a convex function in $x_i$, a fact that follows from observing that this term is a scaling of the maximum of two convex functions. Consequently, the necessary and sufficient equilibrium conditions of this game are given by 
$0 \in \nabla_{x_i} c_i(x_i) + \mathbb{E}[r(\omega)(X+x_i)-d(\omega)]  + \partial_{x_i} h_i(x_i) + \mathcal{N}_{\mathcal{X}_i}(x_i)$ for $i = 1, \hdots, {\bf N}$ 
where $h_i(\bullet)$, defined as $h_i(x_i) \triangleq a_i^\mathsf{T}  \max\{Q_i^{-1}b_i(x_i), l_i(x_i)\}$,  is a convex function in $x_i$. Then the necessary and sufficient equilibrium conditions are given by
\begin{align}\tag{SGE$_{\rm mlf}$}
	0 &\in T(x) \triangleq A(x)+B(x), \\ \mbox{ where } A(x) &\triangleq  G(x) + R(x),  B(x) \triangleq D(x) + \mathcal{N}_{\mathcal X}(x).
\notag
\end{align}
Here $G(x) \triangleq (c_i'(x_i))_{i=1}^N$, $R(x) \triangleq \mathbb{E}[r(\omega)(X{\bf 1}+x) -d(\omega) {\bf 1}] $ and
$D(x) \triangleq \mathbb{E}[\partial_{x_i} h_i(x_i)]$. 
We observe that $G$ is a monotone map while $D(x)$ is the Cartesian product of the expectations of subdifferentials of convex functions, implying that $D$ is also monotone. Furthermore, $R$ is monotone since $\nabla_x R (x) = \mathbb{E}[r(\omega)(\mathbf{I}+{\bf 1}{\bf 1}^\mathsf{T} )]\succeq 0.$ Since $\mathcal{N}_{\mathcal X}$ is a normal cone of a convex set, it is also a monotone map, implying that $T$ is monotone.   

\noindent {\em (b) Model predictive control (MPC) with probabilistic and risk constraints.}
{Model-predictive control (MPC) is a framework for the
control of complex systems~\cite{mpc1}. It obviates  the challenging
derivation/computation of a feedback control law with repeated resolution
of a finite-horizon constrained optimization problem. Contending with uncertainty has prompted the
development of several approaches: (i) {\em Robust approaches.} Robust frameworks for MPC~\cite{mpc2,mpc3,mpc4,cl1,cl2,cl3} often require 
bounded and deterministic descriptions of uncertainty, a property inherited
from robust optimization~\cite{ben2009robust}; (ii) {\em Probabilistic framework.}  Under a probabilistic representation of
the uncertainty, chance-constrained MPC framework~\cite{smpc1,smpc2,smpc4,smpc5,smpc6,smpc7} can be
adopted, allowing for shaping the probability distribution of
system states. Such
avenues have assumed relevance in settings such as climate
control, process control, power systems operation, and vehicle path planning
(cf.~\cite{smpcreview} for an excellent survey). 
}
Suppose the dynamics are captured by a linear discrete-time system,  defined as  
\begin{equation} \label{eq:linear_system}
    \x_{t+1}=A(\delta) \x_t + B_u(\delta) \u_t + B_w(\delta) \w_t,
\end{equation}
where $\x_0$ is given, $\x_t\in \Xscr$ denotes the state of the
system at time $k$, $\u_t \in \Uscr \subseteq \Real^m$ represents the
control input vector at time $t$, and $\w_t \in \Real^p$ is an unmeasurable
disturbance signal at time $t$. In addition, $\Xscr$ and $\Uscr$ represent the set of
states and controls, respectively while the random matrices $A(\delta) , B_u
(\delta), B_w(\delta)$ lie in $\Real^{n \times n}, \Real^{n \times m}$, and
$\Real^{n \times p}$, respectively. We assume access to the distributions
governing $\delta$ and $\w_k$.  Suppose $\bm{\pi} \triangleq \{\pi_0(\cdot), \cdots,
\pi_{N-1}(\cdot)\}$ represents feedback-control policy where $\pi_i: \Real^n
\to \mathcal{U} \subseteq \Real^m$ denotes the state feedback control law for
$i = 0,1, \cdots, N-1$. We may then formally define the value function
$\mathcal{V}_N$ as   $\mathcal{V}_N(\x,\bm{\pi}) \triangleq
\mathbb{E}_{\x_0}\left[ \sum_{i=1}^N J(\x_i,\u_i) + J_N(\x_N)\right]$ where $\mathbb{E}_{\x_0}[\bullet] \triangleq \mathbb{E}[\bullet \mid \x_0].$ In addition,
    suppose $\Xscr_c$ denotes a set of undesirable outcomes. The resulting chance-constrained stochastic control problem requires determining the feedback-control law $\bm{\pi}$ that minimizes $\mathcal{V}_N(\x,\bm{\pi})$ subject to the prescribed dynamics and probabilistic requirements on the state. This problem is challenging, motivating the construction of a finite-horizon open-loop counterpart. To this end, we define $\x_{t,T}$ and
$\u_{t,T}$ as $\x_{t,T} \triangleq \{\x_t,\cdots,\x_{t+T-1)}\}$ and $\u_{t,T} \triangleq \{\u_t,\cdots,\u_{t+T-1)}\}$, respectively while the finite-horizon value function at the $t$-th step looking $T$ periods ahead,  denoted by $\mathcal{V}_{t,T}(\x_{t,T}, \u_{t,T})$, is defined as $\mathcal{V}_{t,T}(\x_{t,T},\u_{t,T}) \triangleq
\mathbb{E}_{\x_t}\left[ \sum_{i=t+1}^T J(\x_i,\u_i)\right]$ where $T \ll N$. Given a horizon $T$, the resulting  
MPC framework~\cite{camacho2013model} requires minimizing  $\mathcal{V}_{t,T}(\x_{t,T}, \u_{t,T})$ subject to the prescribed dynamics and the probabilistic state-constraints, given $\x_t$.  A formal definition of the chance-constrained stochastic control problem ({\bf CC-SC}) and its finite-horizon counterpart ({\bf CC-MPC$_{t}^T$}) is  provided next. 

    {\tiny 
\begin{align*} 
\fbox{$\begin{aligned}
        & \ \mbox{\bf (CC-SC)} \max_{\bm{\pi},\x}  \  \mathcal{V}_N(\x,\bm{\pi}) \\
\mbox{s.t.} & \ \mathbb{P}_{\x_0}\left\{ \omega :  \x_{i} \in  \Xscr_c, \forall i \in [1;N]  \right\} \leq \epsilon \\
    & \ \text{ where $\x_k$ evolves as per  \eqref{eq:linear_system}}. 
\end{aligned}$} \  
\fbox{$\begin{aligned} & \, \mbox{\bf (CC-MPC$_{t}^T$)} \ 
        \max_{\tiny \u_{t,T},\x_{t,T}}  \  \mathcal{V}_{t,T}(\x_{t,T},u_{t,T}) \\
\mbox{s.t.} & \ \mathbb{P}_{\x_t}\left\{ \omega :  \x_{i} \in  \Xscr_c, \forall i\in [k;T]  \right\} \leq \epsilon \\
    & \ \text{ where $\x_t$ evolves as per \eqref{eq:linear_system}} 
\end{aligned}$}.
\end{align*}
}
The control decision $\u_t$ is obtained from resolving ({\bf CC-MPC$_{t,T}$}) and is  
then applied to the system {after which the window is moved ahead. The
resulting  problem ({\bf CC-MPC$_{t+1,T}$}) is then 
resolved} when $t+T < N$ (alternately, the horizon $T$ is reduced appropriately). {This formulation is relatively flexbile and
and can be used to address diverse types of objectives and constraints.} \uss{In general, the problem (CC-MPC$_t^T$) is challenging, owing to the presence of the chance constraint. The probability function can be recast as an expectation of an indicator function over a set but this leads to discontinuous integrands. Recently, the second author has developed avenues where under prescribed assumptions under which the following holds~\cite{bardakci2018probability}. 
\begin{align}\label{equiv-P}
\mathbb{P}\left[\zeta \in \mathcal{K} \mid \zeta \in {\bf K}(\x)\right] = \mathbb{E}_{\xi}\left[ F(\x,\xi)\right], 
\end{align}
where $\mathcal{K}$ is a set in $\Real^n$ symmetric about the origin, ${\bf K}(\x)$ is defined as 
${\bf K}(\x)\triangleq \{\zeta :c(\x,\zeta) \geq 0\}$, $T\in\Real^{d\times n}$, and    
\begin{align*}
c(\x,\zeta)\triangleq \left\{
\begin{aligned}
&1-|\zeta^\mathsf{T}\x|^m, \quad \mbox{Setting A} \\
&T\x-\zeta. \quad \qquad \mbox{Setting B}
\end{aligned} \right.
\end{align*}}
\uss{The integrand $F(\bullet,\xi)$ is defined appropriately in Settings A and B where
in each case, it is shown that $\xi$. In fact, we can then show that a
composition of $\mathbb{E}[F(\bullet,\xi)]$ is convex; e.g. in Setting A,
$1/\mathbb{E}[F(\bullet,\xi)]$ is convex. For expository ease, we may recast  (CC-MPC$_t^T$) as the following chance-constrained problem (CCP) and provide its necessary and sufficient optimality conditions in (SGE$_{\rm ccp}$).  

    {\scriptsize 
\begin{align*} 
\fbox{$\begin{aligned}
        & \ \mbox{\bf (CCP)} \min_{\x}  \mathbb{E}_{\omega}[h(\x,\omega)]\\
& \mbox{s.t.}  \ g(\x) \leq 0,   \quad (\lambda)  
\end{aligned}$} \quad \overset{\mbox{(Reg. conds)}}{\equiv} \quad   
\fbox{$\begin{aligned} & \, \mbox{\bf (SGE$_{\rm ccp}$)} \quad \mbox{Find $\z\triangleq (x,\lambda)$ such that} \\ 
& 0 \in T(\z) \triangleq H(\z)+\mathcal{N}_\mathcal{Z}(\z),
\end{aligned}$}.
\end{align*}
}
where $g(\x) \triangleq \tfrac{1}{\mathbb{E}_{\xi}[F(\x,\xi)]} - \tfrac{1}{(1-\epsilon)}$,  $\mathcal{Z}\triangleq \mathcal{X}\times\Real^+_m$, $\mathcal{N}_\mathcal{Z}(z)$ denotes normal cone of $\mathcal{Z}$ at $z$, $H$ is a monotone set-valued map defined as
\begin{align*}
H(\x,\lambda)\triangleq \left\{\nabla_x f(\x)+\lambda\partial_xg(\x)\right\} \times \{ g(\x) \}.
\end{align*}}

We close by noting that monotone inclusions with expectation-valued operators are of crucial relevance in decision and control problems, providing a strong motivation for addressing their tractable resolution. 
\subsection{Related work.} We provide a brief review of prior research.
\begin{table}
\begin{center}
\begin{threeparttable} 
\caption{Variance-reduced vs Stochastic operator-splitting schemes for SGEs}
\label{tab:split}
\tiny
	\begin{tabular}[t]{  c | C{.8cm}  | c | c | c  } 
	\hline Alg/Prob. & $A, B$ & Biased & $\gamma_k$; $N_k$ & Statements  \\ \hline 
	\cite{rosasco2016stochastic} &  $L_A$, MM &  N  &   SS, NS; 1 & {\tiny \tabincell{l}{$x_k \xrightarrow[a.s]{k \to \infty} x^*\in \mathcal{X}^*$}}  \\   
\hline
	\cite{rosasco2016stochastic} &  $L_A$, $\sigma_A$, $\mu_B$  &  N &   SS, NS; 1 & {\tiny \tabincell{l}{$\mathbb{E}[\|x_k-x^*\|^2] \leq \mathcal{O}(1/k)$}}  \\   
\hline
	\cite{combettes2016stochastic} &  $L_A$,  MM  &  N   &   SS, NS; 1 & {\tiny \tabincell{l}{$x_k \xrightarrow[a.s]{k \to \infty} x^*\in \mathcal{X}^*$}}  \\  
\hline
 {\bf (vr-SMFBS)} &  $\sigma_A$, MM & Y & \tabincell{c}{$\gamma$;$\lfloor \rho^{-(k+1)}\rfloor$\\ $\rho < 1$}   & {\tiny \tabincell{l}{$x_k \xrightarrow[a.s]{k \to \infty} x^*$ \\ $\mathbb{E}[\|x_k-x^*\|^2] \leq \mathcal{O}(q^k)$\\
 Sample-complexity: $\mathcal{O}(1/\epsilon)$}} \\
\hline 
 {\bf (vr-SMFBS)} &  $L_A$, MM & N  &  \tabincell{c}{$\gamma$;$\lfloor k^{a}\rfloor$\\ $a > 1$}   & {\tiny \tabincell{l}{$x_k \xrightarrow[a.s]{k \to \infty} x^* \in \mathcal{X}^*$ \\  $\mathbb{E}[G(\bar{x}_k)] \leq \mathcal{O}(1/k)$\\
 Sample-complexity: $\mathcal{O}(1/\epsilon^{2+\delta})$}}  \\
 \hline
	\end{tabular}

    \begin{tablenotes}
    \centering
\item {\scriptsize \texttt{$L_A$: Lipschitz constant of $A$, MM: Maximal monotone}}
\item {\scriptsize \texttt{$\sigma_A,\mu_B$: strong monotonicity constants of $A$ and $B$}} 
\item {\scriptsize \texttt{SS: square-summable, NS: non-summable}}
    \end{tablenotes}
  \end{threeparttable}
\vspace{-0.3in}
\end{center}
\end{table}

\noindent {\em (a) Stochastic operator splitting schemes.} 
The regime where the maps are expectation-valued has seen
relatively less study~\cite{ruszczynski1997decomposition}.
Stochastic proximal gradient
schemes~\cite{scholschmidt2011convergence,ghadimi15,jalilzadeh2018smoothed,jofre2019variance,rosasco2019convergence}
are an instance of stochastic operator splitting techniques where
$A$ is either the gradient or the subdifferential operator. In the context of monotone inclusions, when
$A$ is a more general monotone operator, a.s. convergence of the
iterates has been proven in ~\cite{combettes2016stochastic}  and
~\cite{rosasco2016stochastic} when $A$ is Lipschitz and
expectation-valued while $B$ is maximal monotone. In fact, in our prior work~\cite{cui16analysis}, we prove a.s. convergence and derive an optimal rate in terms of the gap function in the context of stochastic variational inequality problems with structured operators. Stability analysis~\cite{yaji20analysis} and two-timescale variants~\cite{yaji20stochastic} have also been examined. A rate
statement of $\mathcal{O}(\tfrac{1}{k})$ in terms of mean-squared
error has also been provided when $A$ is additionally strongly
monotone in~\cite{rosasco2016stochastic}. A more general problem which finds a
zero of the sum of three maximally monotone operators is proposed
in~\cite{cevher2018stochastic}. A comparison of rate statements for stochastic
operator-splitting schemes is provided in Table~\ref{tab:split} from which
we note that ({\bf vr-SMBFS}) is equipped with deterministic (optimal) rate
statements, optimal or near-optimal sample-complexity, a.s. convergence
guarantees, and does not require imposing a conditional unbiasedness
requirement on the oracle in the strongly monotone regime. We believe our rate
statements are amongst the first in maximal monotone settings (to the best
of our knowledge).

\noindent {\em (b) Other related schemes.} A natural approach for
resolving SGEs is via sample-average
approximation~\cite{chen2012stochastic,shapiro2008stochastic}. It
proves the weak a.s. convergence and establishes the rate
of convergence in expectation under strong monotonicity
assumption.

\subsection{Gaps and resolution.}
\begin{enumerate}

    \item[(i)] {\em Poorer empirical performance when resolvents are costly.}
Deterministic schemes for strongly monotone and monotone generalized equations
display linear and $\mathcal{O}(1/k)$ rate in resolvent operations
while stochastic analogs display rates of $\mathcal{O}(1/k)$ and
$\mathcal{O}(1/\sqrt{k})$, respectively. This leads to far poorer
practical behavior particularly when the resolvent is challenging to compute,
e.g., in strongly monotone regimes, the complexity in resolvent operations can increase
from $\mathcal{O}(\log(1/\epsilon))$ to $\mathcal{O}(1/\epsilon)$. The proposed scheme ({\bf vr-SMBFS}) achieve
deterministic rates of convergence with either identical or
slightly worse oracle complexities in both monotone and
strongly monotone regimes, allowing for run-times comparable to deterministic counterpart.

\item[(ii)] {\em Absence of rate
statements for monotone operators.} To the
best of our knowledge, there appear to be no non-asymptotic
rate statements available in monotone regimes. In ({\bf vr-SMBFS}), rate statements are now provided.

\item[(iii)]  {\em
Biased oracles.} In many settings, conditional unbiasedness
of the oracle may be harder to impose and one may need to
impose weaker assumptions.  Our proposed scheme allows for
possibly biased oracles in some select settings.

\item[(iv)] {\em State-dependent bounds
on subgradients and second moments.} Many subgradient and
stochastic approximation schemes impose bounds of the form
$\mathbb{E}[\|G(x,\omega)\|^2] \leq M^2$ where $G(x,\omega)
\in T(x,\omega)$ or $\mathbb{E}[\|w\|^2 \mid x]
\leq \nu^2$ where $w = \nabla_x f(x,\omega) - \nabla_x
f(x)$. Both sets of assumptions are often challenging to
impose non-compact regimes. Our scheme can accommodate state-dependent
bounds to allow for non-compact domains.

\end{enumerate}

\subsection{Outline and contributions.}   
We now articulate our contributions. 
In
Section~\ref{sec:split}, we consider the resolution of monotone inclusions in structured regimes where
the map can be expressed as the sum of two maps, facilitating the
use of splitting. In this context, when one of
the maps is expectation-valued while the other has a cheap
resolvent, we consider a scheme where a sample-average of the
expectation-valued map is utilized in the forward step. When the sample-size is increased at a suitable rate, the sequence of iterates is shown to converge a.s. to a solution of the constrained stochastic generalized equation in both monotone and strongly monotone regimes. In addition,  the
resulting sequence of iterates converges either at a linear rate (strongly
monotone) or at a rate of $\mathcal{O}(1/k)$ (maximal monotone), leading
to optimal oracle complexities of $\mathcal{O}(1/\epsilon)$ and
$\mathcal{O}(1/\epsilon^{2+\delta})$ ($\delta>0$), respectively. Notably, the strong monotonicity claim is made without an unbiasedness requirement on the oracle while weaker state-dependent noise requirements are assumed througout. Rate statements in maximally monotone regimes rely on using the Fitzpatrick gap function for inclusion problems. We believe that the rate statements in monotone regimes are amongst the first.
In addition, we provide some background in Section~\ref{sec:back} while preliminary numerics are presented in Section~\ref{sec:numerics}. 

\subsection{Comments on variance-reduced schemes.}   
Before proceeding, we briefly digress regarding the term {\em variance-reduced.}  

\noindent (i) {\em Terminology and applicability.} The moniker
``variance-reduced'' reflects the usage of increasing accurate approximations
of the expectation-valued map, as opposed to noisy sampled variants that are
used in single sample schemes. The resulting schemes are often referred to as
{\em mini-batch} SA schemes and often achieve deterministic rates of
convergence. This work draws inspiration from the early work by Friedlander and
Schmidt~\cite{FriedlanderSchmidt2012} and Byrd et al.~\cite{byrd12} which demonstrated how increasing
sample-sizes can enable achieving deterministic rates of convergence. Similar
findings regarding the nature of sampling rates have been presented
in~\cite{pasupathy2018sampling}. This avenue has proven particularly useful in
developing accelerated gradient schemes for smooth~\cite{jofre2019variance,ghadimi2016mini} and nonsmooth~\cite{jalilzadeh2018smoothed} convex/nonconvex stochastic optimization, variance-reduced quasi-Newton schemes~\cite{raghu17,jalilzadeh20variable}, amongst  others.  Schemes such as
SVRG~\cite{NIPS2013_ac1dd209} and SAGA~\cite{NIPS2014_ede7e2b6} also achieve
deterministic rates of convergence but are customized for finite sum problems
unlike mini-batch schemes that can process expectations over general
probability spaces.  Unlike in mini-batch schemes where increasing batch-sizes
are employed, in schemes such as SVRG, the entire set of samples is
periodically employed for computing a step. 

\noindent (ii) {\em Weaker assumptions and stronger statements.} The proposed
variance-reduced framework has several crucial benefits that can often not be reaped
in the single-sample regime. For instance, rate statements are derived in monotone regimes which have hitherto been unavailable. Second, a.s. convergence guarantees are obtained and in some cases require far weaker moment assumptions. Finally, since the schemes allow for deterministic rates, this leads to far better practical behavior as the numerics reveal. 

\noindent (iii) {\em Sampling requirements.} Naturally, variance-reduced
schemes can generally be employed only when sampling is relatively cheap compared to the main
computational step (such as computing a projection or a prox.) In terms of
overall sample-complexity, the proposed schemes are near optimal. As $k$ becomes large,
one might question how one might contend with $N_k$ tending to $+\infty$. This
issue does not arise since most schemes of this form are meant to provide
$\epsilon$-approximations. For instance, if $\epsilon = 10^{-3}$, then such a scheme
requires approximately $\mathcal{O}(10^{3})$ steps (in monotone settings). Since $N_k \approx \lfloor
k^a \rfloor$ and $a > 1$, we require approximately $(\mathcal{O}(10^3))^{1+a}$
samples in total. In a setting where multi-core architecture is ubiquitous, such
requirements are not terribly onerous particularly since computational costs
have been reduced from $\mathcal{O}(10^6)$ (single-sample) to
$\mathcal{O}(10^3)$. It is worth noting that finite-sum problems routinely have $10^9$ or more samples and competing schemes such as SVRG would require taking the full batch-size intermittently, which means they use $\mathcal{O}(10^9)$ samples at least to achieve the same accuracy as our scheme.

\section{Background}\label{sec:back}
In this section, we provide some background on splitting schemes, building a foundation for the
subsequent sections.  \\
Consider the generalized equation 
\begin{align}\tag{GE}
 0 \in T(x) \triangleq A(x)+B(x). \label{GE}
\end{align} 
If
the resolvent of either $A$ or $B$ (or both) is tractable,
then splitting schemes assume relevance. Notable instances include Douglas-Rachford splitting~\cite{douglas1956numerical,lions1979splitting}, Peaceman-Rachford splitting~\cite{peaceman1955numerical,lions1979splitting}, and Forward-Backward splitting (FBS)~\cite{lions1979splitting,passty1979ergodic}.   
\noindent (a) {\em Douglas-Rachford Splitting}~\cite{douglas1956numerical,lions1979splitting}. In this scheme, the resolvent of $A$ and $B$ can be separately evaluated to generate a sequence defined as follows.
\begin{align} \tag{DRS}
\begin{aligned}
x_{k+\frac{1}{2}} &\coloneqq (\mathbf{I}+\gamma A)^{-1}(x_k) \\
x_{k+1} &\coloneqq(\mathbf{I}+\gamma B)^{-1}(2x_{k+\frac{1}{2}}-x_k) + x_k - x_{k+\frac{1}{2}}.
\end{aligned}
\end{align}
\noindent(b) {\em Peaceman-Rachford Splitting}~\cite{peaceman1955numerical,lions1979splitting}. In contrast, in the Peaceman-Rachford splitting method, the the roles of $A$ and $B$ are exchanged in each iteration, given by the following. 
\begin{align} \tag{PRS}
\begin{aligned}
\notag x_{k+\frac{1}{2}} &\coloneqq (\mathbf{I}+\gamma B)^{-1}(\mathbf{I}-\gamma A)(x_k), \\
\notag x_{k+1} &\coloneqq(\mathbf{I}+\gamma A)^{-1}(\mathbf{I}-\gamma B)(x_{k+\frac{1}{2}}).
\end{aligned}
\end{align}
\noindent (c) {\em Forward-backward splitting}~\cite{lions1979splitting,passty1979ergodic}.  Moreover, if the resolvent of $B$ is easier to evaluate and $A$ and $B$ are maximal monotone, the forward-backward splitting method~\cite{lions1979splitting,passty1979ergodic} was applied to convex optimization in \cite{han1988parallel}:
\begin{align} \tag{FBS}
\notag x_{k+1}\coloneqq(\mathbf{I}+\gamma B)^{-1}(\mathbf{I}-\gamma A)(x_k).
\end{align}
In~\cite{combettes2016stochastic}, a 
stochastic variant of the FBS method,  developed for strongly monotone maps, is equipped with a rate of $\mathcal{O}(1/k)$  while in~\cite{rosasco2016stochastic}, maximal monotone regimes are examined and a.s. convergence statements are provided. A drawback of (FBS) is the requirement of either a
strong monotonicity assumption on $A^{-1}$, or that $A$ be Lipschitz continuous on $\mbox{dom}(A)=\mathbb{R}^n$ and $T$ be strongly monotone; this motivated the modified FBS scheme where convergence was proven when $A$ is monotone and Lipschitz~\cite{tseng2000modified}.
\begin{align} \tag{MFBS}
\begin{aligned}
\notag x_{k+\frac{1}{2}}&\coloneqq(\mathbf{I}+\gamma B)^{-1}(\mathbf{I}-\gamma A)(x_k), \\
\notag x_{k+1}&\coloneqq x_{k+\frac{1}{2}}-\gamma (A(x_{k+\frac{1}{2}})-A(x_k)).
\end{aligned}
\end{align}
In Section~\ref{sec:split}, we develop a variance-reduced {stochastic}
MFBS scheme where $A$ is Lipschitz and monotone, $A(x) \triangleq
\mathbb{E}[A(x,\omega)]$, and $B$ is maximal monotone with a tractable
resolvent; we derive linear and sublinear convergence under strongly
monotone and merely monotone $A$, respectively, achieving deterministic rates
of convergence.

\section{Stochastic Modified Forward-Backward Splitting Schemes}\label{sec:split}
In this section we analyze stochastic (operator) splitting
schemes.
In the case where $A(x) \triangleq \mathbb{E}[A(x,\omega)]$, $A: \mathbb{R}^n \times \Omega \to \mathbb{R}^n$, and $B$ has
a cheap resolvent, we 
develop a variance-reduced splitting
framework. In Section~\ref{sec:4.1}, we provide some
background and outline the assumptions and derive
convergence theory for monotone and strongly monotone
settings in Section~\ref{sec:4.2} and ~\ref{sec:4.3},
respectively.

\subsection{Background and assumptions}\label{sec:4.1}
Akin to other settings that employ stochastic approximation, we assume the presence of a stochastic first-order oracle for operator $A$ that produces a sample $A(x,\omega)$ given a vector $x$. Such a sample is core to developing a variance-reduced modified forward-backward
splitting ({\bf vr-SMFBS}) scheme reliant on 
$\tfrac{\sum_{j=1}^{N_k} A(x_k,\omega_{j,k})}{N_k}$ to
approximate $\mathbb{E}[A(x_k,\omega)]$ at iteration $k$. Given an $x_0 \in \mathbb{R}^n$, we
formally define such a scheme next.
\begin{align} \tag{\bf vr-SMFBS}
\begin{aligned}
\notag x_{k+\frac{1}{2}}&\coloneqq(\mathbf{I}+\gamma B)^{-1}(x_k-\gamma A_k), \\
\notag x_{k+1}&\coloneqq x_{k+\frac{1}{2}}-\gamma (A_{k+\frac{1}{2}}-A_k),
\end{aligned}
\end{align}
where $A_k \triangleq \tfrac{\sum_{j=1}^{N_k}A(x_k,\omega_{j,k})}{N_k}$,
$A_{k+\frac{1}{2}} \triangleq
\tfrac{\sum_{j=1}^{N_k}A(x_{k+{1}/{2}},\omega_{j,{k+{1}/{2}}})}{N_k}$ are
estimators of $A(x_k)$ and $A(x_{k+\frac{1}{2}})$, respectively.  We assume the following on operators $A$ and $B$.
\begin{assumption}\label{lip-mon} 
The operator $A$ is single-valued, monotone and $L$-Lipschitz on $\mathbb{R}^n$, i.e., $\forall x, y \in \mathbb{R}^n$, 
$\|A(x)-A(y)\|  \le L\|x-y\|$
 and $(A(x)-A(y))^\mathsf{T}(x-y)  \ge0$; the operator $B$ is maximal monotone on $\mathbb{R}^n$. 
\end{assumption}
\vspace{0.1in}

Suppose $\mathcal{F}_k$ denotes  the history up to iteration $k$, i.e., $\mathcal{F}_0  = \{x_0\}$, 
{\tiny
\begin{align*} 
\mathcal{F}_k  & \triangleq
\left\{x_0,\{A(x_0,\omega_{j,0})\}_{j=1}^{N_0},\{A(x_{\frac{1}{2}},\omega_{j,\frac{1}{2}})\}_{j=1}^{N_0},  \cdots, \{A(x_{k-1},\omega_{j,k-1})\}_{j=1}^{N_{k-1}}, \right. \\
	&  \left.   \{A(x_{k-\frac{1}{2}},\omega_{j,k-\frac{1}{2}})\}_{j=1}^{N_{k-1}}\right\},  
\mathcal{F}_{k+\frac{1}{2}} \triangleq \mathcal{F}_k\cup\{A(x_k,\omega_{j,k})\}_{j=1}^{N_k}.
\end{align*}}

Suppose $w_k\triangleq
A(x_k,\omega_k)-A(x_k)$, $\bar{w}_k\triangleq
\tfrac{\sum_{j=1}^{N_k}(A(x_k,\omega_{j,k})-A(x_k))}{N_k}$,
$w_{k+\frac{1}{2}}\triangleq
A(x_{k+\frac{1}{2}},\omega_{k+\frac{1}{2}})-A(x_{k+\frac{1}{2}})$ and
$\bar{w}_{k+\frac{1}{2}}\triangleq
\tfrac{\sum_{j=1}^{N_k}(A(x_{k+\frac{1}{2}},\omega_{j,{k+\frac{1}{2}}})-A(x_{k+\frac{1}{2}}))}{N_k}$, where $N_k$ denotes the batch-size of samples $A(x,\omega_{j,k})$ at iteration $k$. We impose the following bias and moment assumptions on $\bar{w}_k$ and $\bar{w}_{k+1/2}$. Note that Assumption~\ref{moment}(ii) is weakened in the strongly monotone regime, allowing for biased oracles.

\begin{assumption}\label{moment} 
At iteration $k$, the following hold in an a.s. sense:
(i) The conditional means $\mathbb{E}[w_k\mid\mathcal{F}_k]$ and
$\mathbb{E}[w_{k+\frac{1}{2}}\mid\mathcal{F}_{k+\frac{1}{2}}]$ are zero
for all $k$ in an a.s. sense; (ii) The conditional second moments are bounded 
in an a.s. sense as follows, i.e. there exists $\nu_1, \nu_2$ such that $\mathbb{E}[\|w_k\|^2\mid\mathcal{F}_k]\le\nu_1^2\|x_k\|^2 + \nu_2^2$ and
$\mathbb{E}[\|w_{k+{1}/{2}}\|^2\mid\mathcal{F}_{k+{1}/{2}}]\le\nu_1^2\|x_{k+1/2}\|^2 + \nu_2^2$
for all $k$ in an a.s. sense. 
\end{assumption}

When the feasible set $\mathcal{X}$ is possibly unbounded, the assumption that
the conditional second moment $w_k$ is uniformly bounded
a.s. is often a stringent requirement. Instead, we impose a
state-dependent assumption on $w_k$. We conclude this subsection by defining a residual function for a generalized equation. 
\begin{lemma}[{\bf Residual function for (GE)}] \label{rf} \em Suppose $\gamma > 0$, $T = A+B$, and  
\begin{align*}
r_\gamma(x) \triangleq \|x-(\mathbf{I}+\gamma B)^{-1}(x - \gamma A(x))\|. 
\end{align*}
Then $r_\gamma$ is a residual function for (GE). 
\end{lemma}
\begin{proof}
By definition, $r_\gamma(x)=0$ if and only if  $x=(\mathbf{I}+\gamma B)^{-1}(x - \gamma A(x)).$ This can be interpreted as follows, leading to the conclusion that $x \in T^{-1}(0)$.\begin{align*}
&x=((\mathbf{I}+\gamma B)^{-1}(x - \gamma A(x)) \\ &\Leftrightarrow x - \gamma A(x) \in (\mathbf{I}+\gamma B)(x) \Leftrightarrow 0\in A(x)+B(x).
\end{align*}
\end{proof}
We conclude this subsection with two lemmas~\cite{polyak1987introduction} crucial for proving claims of almost sure convergence. 
\begin{lemma}\label{robbins} 
Let $v_k$, $u_k$, $\delta_k$, $\psi_k$ be nonnegative random variables adapted to $\sigma$-algebra $\mathcal{F}_k$, and let the following relations hold almost surely:
\begin{align*}
&\mathbb{E}[v_{k+1}\mid\mathcal{F}_k]\leq(1+u_k)v_k-\delta_k+\psi_k, \quad\forall k; \\ &\quad \sum_{k=0}^{\infty}u_k<\infty,\mbox{ and } \sum_{k=0}^{\infty}\psi_k<\infty.
\end{align*}
Then a.s., 
$\lim_{k\to\infty}v_k=v$ and $\sum_{k=0}^{\infty}\delta_k<\infty,$
where $v\ge0$ is a random variable. 
\end{lemma}

\begin{lemma}\label{as-recur} \em Consider a sequence $v_k$ of nonnegative random variables adapted to the $\sigma$-algebra $\mathcal{F}_k$ and satisfying 
$ \mathbb{E}[v_{k+1} \mid \mathcal{F}_k] \leq (1-a_k) v_k + b_k$ for $k \geq 0$
where $a_k \in [0,1], b_k \geq 0$ for every $k \geq 0$ and 
$ \sum_{k=0}^{\infty} a_k = \infty, \quad \sum_{k=0}^{\infty} b_k < \infty, \mbox{and} \lim_{k \to \infty} \tfrac{b_k}{a_k} = 0. $
Then $v_k \to 0$ a.s. as $k \to \infty$. 
\end{lemma}

\subsection{Convergence analysis under merely monotone $A$} \label{sec:4.2}
In this subsection, we derive a.s. convergence guarantees and rate statements. 
First, we prove the a.s. convergence of the sequence generated by this scheme. We start with a lemma.
\begin{lemma}\label{egt} 
Consider a sequence $\{x_k\}$ generated by ({\bf vr-SMFBS}). Let Assumptions~\ref{lip-mon} and \ref{moment} hold. Suppose $\gamma \le \tfrac{1}{2 \tilde L}$ and $\tilde
L^2 \triangleq L^2+ \tfrac{4 \nu_1^2}{N_0}$.
Then for any $k \ge 0$, 
\begin{align*}
\mathbb{E}[\|x_{k+1}&-x^*\|^2\mid\mathcal{F}_{k}]
\le\left(1+ \tfrac{25\gamma^2 \nu_1^2}{N_k}\right)\|x_k-x^*\|^2 \\&+\tfrac{25\gamma^2\nu_1^2\|x^*\|^2}{N_k}+{\tfrac{17\gamma^2\nu_2^2}{2N_k}} -\tfrac{r^2_\gamma(x_k)}{4}.
\end{align*}
\end{lemma}
\begin{proof}
From the definition of $x_{k+\frac{1}{2}}$ and $x_{k+1}$, we have
\begin{align*} 
\begin{aligned}
x_{k+\frac{1}{2}}&+\gamma v_{k+\frac{1}{2}}=x_k-\gamma (u_k+\bar{w}_k), \\
x_{k+1} & =x_{k+\frac{1}{2}}-\gamma (u_{k+\frac{1}{2}}+\bar{w}_{k+\frac{1}{2}}-u_k-\bar{w}_k).
\end{aligned}
\end{align*}
where $u_k=A(x_k), u_{k+\frac{1}{2}}=A(x_{k+\frac{1}{2}}), v_{k+\frac{1}{2}} \in B(x_{k+\frac{1}{2}})$.
From $0 \in A(x^*)+B(x^*)$, 
\begin{align*}
u^*+v^*=0, \mbox{ where }  u^*=A(x^*), \quad v^* \in B(x^*)
\end{align*}
We have the following equality:
\begin{align}
\notag & \quad \|x_k-x^*\|^2=\|x_k-x_{k+\frac{1}{2}}+x_{k+\frac{1}{2}}-x_{k+1}+x_{k+1}-x^*\|^2\\
\notag&=\|x_k-x_{k+\frac{1}{2}}\|^2+\|x_{k+\frac{1}{2}}-x_{k+1}\|^2+\|x_{k+1}-x^*\|^2\\\notag&+2(x_k-x_{k+\frac{1}{2}})^\mathsf{T}(x_{k+\frac{1}{2}}-x^*)\\\notag&+2(x_{k+\frac{1}{2}}-x_{k+1})^\mathsf{T}(x_{k+1}-x^*)\\
\notag&=\|x_k-x_{k+\frac{1}{2}}\|^2+\|x_{k+\frac{1}{2}}-x_{k+1}\|^2+\|x_{k+1}-x^*\|^2\\\notag&+2(x_k-x_{k+\frac{1}{2}})^\mathsf{T}(x_{k+\frac{1}{2}}-x^*)-2\|x_{k+\frac{1}{2}}-x_{k+1}\|^2\\\notag&+2(x_{k+\frac{1}{2}}-x_{k+1})^\mathsf{T}(x_{k+\frac{1}{2}}-x^*)\\
\notag&=\|x_k-x_{k+\frac{1}{2}}\|^2-\|x_{k+\frac{1}{2}}-x_{k+1}\|^2+\|x_{k+1}-x^*\|^2\\\notag&+2(x_k-x_{k+1})^\mathsf{T}(x_{k+\frac{1}{2}}-x^*)\\
\notag&=\|x_k-x_{k+\frac{1}{2}}\|^2 -\gamma^2\|u_{k+\frac{1}{2}}+\bar{w}_{k+\frac{1}{2}}-u_k-\bar{w}_k\|^2\\&+\|x_{k+1}-x^*\|^2+2\gamma(u_{k+\frac{1}{2}}+v_{k+\frac{1}{2}}+\bar{w}_{k+\frac{1}{2}})^\mathsf{T}(x_{k+\frac{1}{2}}-x^*).\label{eq18}
\end{align}
By Lemma~\ref{rf}, $r_{\gamma}$ is a residual function for (GE), defined as  $r_\gamma(x) \triangleq \|x-((\mathbf{I}+\gamma B)^{-1}(x - \gamma A(x))\|$. It follows that
\begin{align*}
&r_{\gamma}^2(x_k)  = \|x_k - (\mathbf{I}+\gamma B)^{-1}(x_k - \gamma A(x_k))\|^2 \\
	& =\|x_k-x_{k+\frac{1}{2}}+ x_{k+\frac{1}{2}} -(\mathbf{I}+\gamma B)^{-1}(x_k-\gamma A(x_k))\|^2 \\\notag
	& \le2\|x_k-x_{k+\frac{1}{2}}\|^2\\&+2\| (\mathbf{I}+\gamma B)^{-1}(x_k-\gamma A_k) -(\mathbf{I}+\gamma B)^{-1}(x_k-\gamma A(x_k))\|^2 \\
\notag&\le2\|x_k-x_{k+\frac{1}{2}}\|^2+2\gamma^2\| \bar{w}_k\|^2,
\end{align*}
where the last inequality holds because $(\mathbf{I}+\gamma B)^{-1}$ is a non-expansive operator.
Consequently, we have that
\begin{align}
-\|x_k-x_{k+\frac{1}{2}}\|^2 & \leq -\tfrac{1}{2}  r_{\gamma}^2(x_k) + \gamma^2\|\bar{w}_k\|^2. \label{eq-s1}
\end{align}
Following \eqref{eq18}, we have
\begin{align}
\notag & \quad \|x_{k+1}-x^*\|^2=\|x_k-x^*\|^2-\|x_k-x_{k+\frac{1}{2}}\|^2\\\notag&+\gamma^2\|u_{k+\frac{1}{2}}+\bar{w}_{k+\frac{1}{2}}-u_k-\bar{w}_k\|^2\\\notag&-2\gamma(u_{k+\frac{1}{2}}+v_{k+\frac{1}{2}}+\bar{w}_{k+\frac{1}{2}})^T(x_{k+\frac{1}{2}}-x^*)\\
\notag&=\|x_k-x^*\|^2-\|x_k-x_{k+\frac{1}{2}}\|^2\\\notag&+\gamma^2\|u_{k+\frac{1}{2}}+\bar{w}_{k+\frac{1}{2}}-u_k-\bar{w}_k\|^2\\
\notag&-2\gamma(u_{k+\frac{1}{2}}+v_{k+\frac{1}{2}})^\mathsf{T}(x_{k+\frac{1}{2}}-x^*)-2\gamma \bar{w}_{k+\frac{1}{2}}^\mathsf{T}(x_{k+\frac{1}{2}}-x^*) \\
\notag&\le\|x_k-x^*\|^2-\|x_k-x_{k+\frac{1}{2}}\|^2+2\gamma_k^2\|u_{k+\frac{1}{2}}-u_k\|^2\\\notag&+2\gamma_k^2\|\bar{w}_{k+\frac{1}{2}}-\bar{w}_k\|^2-2\gamma(u_{k+\frac{1}{2}}+v_{k+\frac{1}{2}})^\mathsf{T}(x_{k+\frac{1}{2}}-x^*)\\\notag&-2\gamma \bar{w}_{k+\frac{1}{2}}^\mathsf{T}(x_{k+\frac{1}{2}}-x^*)\\
\notag&\le\|x_k-x^*\|^2-(1-2\gamma^2L^2)\|x_k-x_{k+\frac{1}{2}}\|^2\\\notag&-2\gamma\underbrace{(u_{k+\frac{1}{2}}+v_{k+\frac{1}{2}})^\mathsf{T}(x_{k+\frac{1}{2}}-x^*)}_{\mtiny{\ge 0}} \\
&+2\gamma^2\|\bar{w}_{k+\frac{1}{2}}-\bar{w}_k\|^2-2\gamma \bar{w}_{k+\frac{1}{2}}^\mathsf{T}(x_{k+\frac{1}{2}}-x^*)\label{new21} \\
\notag& \le \|x_k-x^*\|^2-\left(\tfrac{1}{2}-2\gamma^2L^2\right)\|x_k-x_{k+\frac{1}{2}}\|^2\\\notag&+2\gamma^2\|\bar{w}_k-\bar{w}_{k+\frac{1}{2}}\|^2+2\gamma \bar{w}_{k+\frac{1}{2}}^\mathsf{T}(x^*-x_{k+\frac{1}{2}})\\\notag&-\tfrac{1}{2}\|x_k-x_{k+\frac{1}{2}}\|^2 \\
\notag&\overset{\eqref{eq-s1}}{\leq}\|x_k-x^*\|^2-\left(\tfrac{1}{2}-2\gamma^2L^2\right)\|x_k-x_{k+\frac{1}{2}}\|^2\\\notag&+2\gamma^2\|\bar{w}_k-\bar{w}_{k+\tfrac{1}{2}}\|^2+2\gamma \bar{w}_{k+\frac{1}{2}}^\mathsf{T}(x^*-x_{k+\frac{1}{2}})\\\notag&-\tfrac{1}{4}r^2_\gamma(x_k)+\tfrac{1}{2}\gamma^2\|\bar{w}_k\|^2 \\
\notag&\le\|x_k-x^*\|^2-\left(\tfrac{1}{2}-2\gamma^2L^2\right)\|x_k-x_{k+\frac{1}{2}}\|^2+\tfrac{9}{2}\gamma^2\|\bar{w}_k\|^2\\\notag&+4\gamma^2\|\bar{w}_{k+\frac{1}{2}}\|^2+2\gamma \bar{w}_{k+\frac{1}{2}}^\mathsf{T}(x^*-x_{k+\frac{1}{2}})-\tfrac{1}{4}r^2_\gamma(x_k).
\end{align}
Taking expectations conditioned on $\mathcal{F}_{k}$, we 
obtain the following bound:
\begin{align}
\notag&\mathbb{E}[\|x_{k+1}-x^*\|^2\mid\mathcal{F}_{k}] \le\|x_k-x^*\|^2\\\notag&-(\tfrac{1}{2}-2\gamma^2L^2){\mathbb{E}[}\|x_k-x_{k+\frac{1}{2}}\|^2{\mid\mathcal{F}_{k}]}\\\notag&+\mathbb{E}[\mathbb{E}[4\gamma^2\|\bar{w}_{k+\frac{1}{2}}\|^2\mid\mathcal{F}_{k+\frac{1}{2}}]\mid\mathcal{F}_{k}]+\mathbb{E}\left[\tfrac{9}{2}\gamma^2\|\bar{w}_{k}\|^2\mid\mathcal{F}_{k}\right]\\\notag&-\mathbb{E}[\mathbb{E}[2\gamma \bar{w}_{k+\frac{1}{2}}^\mathsf{T}(x_{k+\frac{1}{2}}-x^*)\mid\mathcal{F}_{k+\frac{1}{2}}]\mid\mathcal{F}_{k}]-\tfrac{1}{4}r^2_\gamma(x_k) \\
\notag & \leq \|x_k-x^*\|^2-(\tfrac{1}{2}-2\gamma^2L^2)\mathbb{E}[\|x_k-x_{k+\frac{1}{2}}\|^2{\mid\mathcal{F}_{k}]}\\\notag&+{\tfrac{4\gamma^2(\nu_1^2 {\mathbb{E}[}\|x_{k+\frac{1}{2}}\|^2{\mid\mathcal{F}_{k}]} + \nu^2_2)}{N_k}+\tfrac{\frac{9}{2}\gamma^2(\nu_1^2 \|x_k\|^2 + \nu^2_2)}{N_k}} 
-\tfrac{1}{4}r^2_\gamma(x_k) \\
\notag & \leq \|x_k-x^*\|^2-(\tfrac{1}{2}-2\gamma^2L^2){\mathbb{E}[}\|x_k-x_{k+\frac{1}{2}}\|^2{\mid\mathcal{F}_{k}]}\\\notag&+\tfrac{4\gamma^2(2\nu_1^2 {\mathbb{E}[}\|x_k-x_{k+\frac{1}{2}}\|^2{\mid\mathcal{F}_{k}]} + 2\nu_1^2 \|x_k\|^2 + \nu^2_2)}{N_k}\\\notag&+\tfrac{\frac{9}{2}\gamma^2(\nu_1^2 \|x_k\|^2 + \nu^2_2)}{N_k}
-\tfrac{1}{4}r^2_\gamma(x_k) \\
\notag & \leq \|x_k-x^*\|^2-(\tfrac{1}{2}-2\gamma^2L^2- \tfrac{8\gamma^2 \nu_1^2}{N_k}){\mathbb{E}[}\|x_k-x_{k+\frac{1}{2}}\|^2{\mid\mathcal{F}_{k}]}\\\notag&+{\tfrac{\tfrac{25}{2}\gamma^2(2\nu_1^2 \|x_k-x^*\|^2 + 2\nu_1^2 \|x^*\|^2)}{N_k}+{\tfrac{17\gamma^2\nu_2^2}{2N_k}}} 
-\tfrac{r^2_\gamma(x_k)}{4} \\
\notag&\le \left(1+ \tfrac{25\gamma^2 \nu_1^2}{N_k}\right)\|x_k-x^*\|^2+\tfrac{25\gamma^2\nu_1^2\|x^*\|^2}{N_k}+{\tfrac{17\gamma^2\nu_2^2}{2N_k}} -\tfrac{r^2_\gamma(x_k)}{4}, 
\end{align}
where the penultimate inequality follows from noting that $\tfrac{1}{2}-2\gamma^2L^2-
\tfrac{8\gamma^2 \nu_1^2}{N_k} \geq \tfrac{1}{2}- 2\gamma^2(L^2+ \tfrac{4
\nu_1^2}{N_0})) \geq 0$, if $\gamma \le \tfrac{1}{2 \tilde L}$ and $\tilde
L^2 \triangleq L^2+ \tfrac{4\nu_1^2}{N_0}.$
\end{proof} 

\begin{theorem}[{\bf a.s. convergence of (vr-SMFBS)}] \label{ass} 
Consider a sequence $\{x_k\}$ generated by {\bf (vr-SMFBS)}. Let Assumptions~\ref{lip-mon} and \ref{moment} hold. Suppose $\gamma \le \tfrac{1}{2 \tilde L}$ and $\tilde
L^2 \triangleq L^2+ \tfrac{4 \nu_1^2}{N_0}$, where
$\{N_k\}$ is a
non-decreasing sequence, and $\sum_{k=0}^\infty \tfrac{1}{N_k} < M$.  Then for any $x_0 \in \Real^n$, $\{x_k\}$ converges to a solution $x^* \in \mathcal{X}^*\triangleq T^{-1}(0)$ in an a.s. sense.
\end{theorem}
\begin{proof}
We may now apply Lemma~\ref{robbins} which allows us to
claim that $\{\|x_k-x^*\|\}$ is convergent for any $x^* \in
\mathcal{X}^*$ and $\sum_{k}r_\gamma(x_k)^2<\infty$ in an a.s.  sense.
Therefore, in an a.s.  sense, we have $$ \lim_{k \to \infty}
r^2_\gamma(x_k) = 0. $$ Since $\{\|x_k-x^*\|^2\}$ is a
convergent sequence in an a.s. sense, $\{x_k\}$ is bounded
a.s. and has a convergent subsequence. Consider any
convergent subsequence of $\{x_k\}$ {with index set} denoted
by ${\cal K}$ and suppose its limit point is denoted by
$\bar{x}$. We have that $\lim_{k \in {\cal K}}
r_{\gamma}(x_k) = r_{\gamma}({\bar x}) = 0$ a.s. since
$r_{\gamma}(\cdot)$ is a continuous function. It follows that
$\bar x$ is a solution to $0\in T(x)$. Consequently, some
convergent subsequence of $\{x_k\}_{k \geq 0}$ , denoted by
$\mathcal{K}$, satisfies $\lim_{k \in \mathcal{K}} x_k = \bar x \in
\mathcal{X}^*$ a.s.. Since $\{\|x_k-x^*\|\}$ is convergent a.s. for any $x^* \in
\mathcal{X}^*$, it follows that $\{\|x_k-\bar{x}\|^2\}$ is convergent a.s.
and its unique limit point is zero. Thus every subsequence
of $\{x_k\}$ converge a.s. to $\bar{x}$ which leads to the
claim that the entire sequence of $\{x_k\}$ is convergent to
a point $\bar{x} \in \mathcal{X}^*$.  
\end{proof}

When the sampling process is computationally expensive (i.e., such as in the queueing systems or PDE, etc.), we prove the following corollary regarding ({\bf vr-SMFBS}) with $N_k = 1$ for every $k$.

\begin{corollary}[{\bf a.s. convergence under single sample}] \label{as-ss}
Consider a sequence $\{x_k\}$ generated by {\bf (vr-SMFBS)}. Let Assumptions~\ref{lip-mon} and \ref{moment} hold. Suppose $\sum_k^\infty\gamma_k=\infty$, $\sum_k^\infty\gamma_k^2<\infty$  and $N_k=1$ for every $k \in \mathbb{Z}_+$. In addition, suppose $A$ is co-coercive with constant $c$ and strictly monotone. Furthermore, suppose $\gamma_k \le \min \left\{\tfrac{1}{4cL^2},\tfrac{1}{\sqrt{2(L^2+4\nu_1^2)}}\right\}$. Then $\{x_{k}\}$ converges to a solution $x^* \in \mathcal{X}^*\triangleq T^{-1}(0)$ in an a.s. sense.
\end{corollary}

To establish the rate under maximal monotonicity,
we need introduce a metric for ascertaining progress. In
strongly monotone regimes, the mean-squared error serves
as such a metric while the
function value represents such a metric in optimization regimes. In merely
monotone variational inequality problems, a special case of
monotone inclusion problems, the
gap function has proved useful (cf.~\cite{larsson94gap,facchinei2007finite}). When considering
the more general monotone inclusion problem, Borwein and
Dutta presented a gap function~\cite{borwein2016maximal},
inspired by the Fitzpatrick
function~\cite{borwein2010convex,fitzpatrick1988representing}.
\begin{definition}[{\bf Gap function}] 
Given a set-valued mapping
$T:\mathbb{R}^n\rightrightarrows \mathbb{R}^n$, then the gap function $G$ associated with the inclusion problem $0 \in T(x)$ is defined as 
$$
G(x) \triangleq \sup_{y\in \mathrm{dom}(T)}\sup_{z\in T(y)}z^\mathsf{T}(x-y),\quad\forall x\in \Real^n. 
$$
\end{definition}
The gap function is nonnegative for all $ x \in \mathbb{R}^n$ and is zero if and
only if $0 \in T(x)$. To derive the convergence rate under maximal monotonicity, we require boundedness of the domain of $T$ as formalized by the next assumption.
\begin{assumption} \label{bd6} 
The domain of $T$ is bounded, i.e. 
$\|x\|\le D_T,\quad\forall x \in \{v \in \mathbb{R}^n \mid T(v) \neq \emptyset\}.$ 
\end{assumption} 
\ic{Clearly, from the definition, a convex gap function can be extended-valued and its domain is contingent on the boundedness properties of dom $T$. When dom $T$ is bounded, the gap function is globally defined but when dom $T$ is unbounded, one resolution is based on the notion of restricted merit functions, first introduced in \cite{nesterov2007dual}. In this approach, the gap function is defined on a bounded set which belongs to dom $T$. In such instances, a local rate of convergence can be obtained.} \\
We begin by establishing an intermediate result. 
\begin{lemma}\label{t} 
Let Assumptions~\ref{lip-mon} and \ref{moment} hold.
Suppose $\{x_k\}$ denotes a sequence generated by
({\bf vr-SMBFS}). Then for all $y \in \mathrm{dom}(T)$, $z\in T(y)$ and all $k\ge0$,
\begin{align}
\notag2&\gamma z^\mathsf{T}(x_{k+\frac{1}{2}}-y)  \le\|x_k-y\|^2-\|x_{k+1}-y\|^2 \\\notag&-(1-2\gamma^2L^2)\|x_k-x_{k+\frac{1}{2}}\|^2+2\gamma^2\|\bar{w}_{k+\frac{1}{2}}-\bar{w}_k\|^2\\\notag&+2\gamma \bar{w}_{k+\frac{1}{2}}^\mathsf{T}(y-x_{k+\frac{1}{2}}).
\end{align} 
\end{lemma}
\begin{proof}
According to \eqref{new21} and replacing $x^*$ with $y\in \mathrm{dom}(T)$, we have that
\begin{align}
\notag2&\gamma z^\mathsf{T}(x_{k+\frac{1}{2}}-y) \le \|x_k-y\|^2-\|x_{k+1}-y\|^2\\\notag&-(1-2\gamma^2L^2)\|x_k-x_{k+\frac{1}{2}}\|^2+2\gamma^2\|\bar{w}_{k+\frac{1}{2}}-\bar{w}_k\|^2\\\notag&-2\gamma \bar{w}_{k+\frac{1}{2}}^\mathsf{T}(x_{k+\frac{1}{2}}-y) \\
\notag&\le \|x_k-y\|^2-\|x_{k+1}-y\|^2-(1-2\gamma^2L^2)\|x_k-x_{k+\frac{1}{2}}\|^2\\&+2\gamma^2\|\bar{w}_{k+\frac{1}{2}}-\bar{w}_k\|^2+2\gamma \bar{w}_{k+\frac{1}{2}}^\mathsf{T}(y-x_{k+\frac{1}{2}}). \label{eq28b}
\end{align}
\end{proof}
Invoking Lemma \ref{t}, we derive a rate statement for $\bar{x}_K$, an
	average of the iterates $\{x_{k+1/2}\}$ generated by ({\bf vr-SMFBS})
over the window constructed from $0$ to $K-1$:
\begin{align}\label{def-ave-smbfs2}
\bar{x}_K\triangleq
\tfrac{\sum_{k=0}^{K-1}x_{k+\frac{1}{2}}}{K}.
\end{align}
\begin{proposition}[{\bf Rate statement under monotonicity}]\label{gapT} 
Consider the ({\bf vr-SMFBS}) scheme. Suppose $x_0 \in \Real^n$ and let $\{\bar{x}_K\}$ be defined in
\eqref{def-ave-smbfs2}. Let Assumptions~\ref{lip-mon} -- \ref{bd6} hold. Suppose $\gamma \le \tfrac{1}{2 \tilde L}$ and $\tilde
L^2 \triangleq L^2+ \tfrac{4 \nu_1^2}{N_0}$,  $\{N_k\}$ is a non-decreasing sequence, and $\sum_{k=0}^\infty N_k<M$.\\
(a) For any $K \ge 1$, $\mathbb{E}[G(\bar{x}_{K})]=\mathcal{O}\left(\tfrac{1}{K}\right).$  \\
(b)  Suppose $N_k=\lfloor k^a\rfloor$, for $a>1$. Then the oracle complexity to compute an $\bar{x}_{K+1}$ such that $\mathbb{E}[G(\bar{x}_{K+1}) \leq \epsilon$ is bounded as 
$\sum_{k=0}^{K}N_k \leq \mathcal{O}\left(\tfrac{1}{\epsilon^{a+1}}\right).$
\end{proposition}
\begin{proof}
(a) 
We first define an auxiliary sequence $\{u_k\}$ such that 
$$ u_{k+1}:= u_k - \gamma \bar{w}_{k+\frac{1}{2}}, $$
where $u_0 \in \Real^n$. We may then express the last term on the right in \eqref{eq28b} as follows.
\begin{align}\notag
 2&\gamma \bar{w}_{k+\frac{1}{2}}^\mathsf{T}({y}-x_{k+\frac{1}{2}})  = 2\gamma \bar{w}_{k+\frac{1}{2}}^\mathsf{T}({y}-u_k) \\\notag&+ 2\gamma \bar{w}_{k+\frac{1}{2}}^\mathsf{T}(u_k-x_{k+\frac{1}{2}}) \\
	\notag& = \|u_k-y\|^2 - \|u_{k+1}-y\|^2 + \gamma^2 \|\bar{w}_{k+\frac{1}{2}}\|^2 \\&+ 2\gamma \bar{w}_{k+\frac{1}{2}}^\mathsf{T}(u_k-x_{k+\frac{1}{2}}). \label{eq-gap1} 
\end{align}
Invoking Lemma~\ref{t} and summing over $k$, we have
\begin{align}
\notag\sum_{k=0}^{K-1}2\gamma z^\mathsf{T}(x_{k+\frac{1}{2}}-y)& \le\|x_0-y\|^2+2\gamma^2\sum_{k=0}^{K-1} \|\bar{w}_k-\bar{w}_{k+\frac{1}{2}}\|^2\\&+2\gamma \sum_{k=0}^{K-1} \bar{w}_{k+\frac{1}{2}}^\mathsf{T}({y}-x_{k+\frac{1}{2}}). \label{eq-s2}
\end{align}
Dividing \eqref{eq-s2} by $K$, we obtain the following. 
\begin{align}
\notag& \quad \tfrac{1}{K} \sum_{k=0}^{K-1} 2\gamma z^\mathsf{T}(x_{k+\frac{1}{2}}-y)\le\tfrac{1}{K}\|x_0-y\|^2+ \\&\tfrac{2\gamma^2\sum_{k=0}^{K-1} \|\bar{w}_k-\bar{w}_{k+\frac{1}{2}}\|^2}{K}
  +\tfrac{\sum_{k=0}^{K-1} 2\gamma \bar{w}_{k+\frac{1}{2}}^\mathsf{T}({y}-x_{k+\frac{1}{2}})}{K}. \label{eq-s3}
\end{align}
Using \eqref{eq-gap1} in \eqref{eq-s3} and invoking \eqref{def-ave-smbfs2}, it follows that 
\begin{align*}
& \quad \gamma z^\mathsf{T}(\bar{x}_K-y)\le\tfrac{1}{2K}\|x_0-y\|^2\\\notag&+\tfrac{2\gamma^2\sum_{k=0}^{K-1} \|\bar{w}_k-\bar{w}_{k+\frac{1}{2}}\|^2}{2 K} 
+ \tfrac{\sum_{k=0}^{K-1} 2\gamma\bar{w}_{k+\frac{1}{2}}^\mathsf{T}({y}-x_{k+\frac{1}{2}})}{2K} \\
&  \le\tfrac{1}{2K}\|x_0-y\|^2+\tfrac{2\gamma^2\sum_{k=0}^{K-1} \|\bar{w}_k-\bar{w}_{k+\frac{1}{2}}\|^2}{2K}\\\notag&+\tfrac{\|u_0-y\|^2 + \sum_{k=0}^{K-1} (\gamma^2 \|\bar{w}_{k+\frac{1}{2}}\|^2+ 2\gamma \bar{w}_{k+\frac{1}{2}}^\mathsf{T}(u_k-x_{k+\frac{1}{2}}))}{2 K}.
\end{align*}
Taking supremum over $z \in T(y)$ and $y \in \mathrm{dom}(T)$ and leveraging the compactness of dom$(T)$,  we obtain the following inequality.
\begin{align*}
 & \quad \gamma \sup_{y \in \mathrm{dom}(T)}\sup_{z \in T(y)} z^\mathsf{T}(\bar{x}_K-y)\\
& \le\tfrac{2D_T^2+\|x_0\|^2+\|u_0\|^2}{K}+\tfrac{\gamma^2\sum_{k=0}^{K-1} (2\|\bar{w}_k-\bar{w}_{k+\frac{1}{2}}\|^2+\|\bar{w}_{k+\frac{1}{2}}\|^2)}{2K}
  \\&+ \tfrac{\gamma\sum_{k=0}^{K-1} \bar{w}_{k+\frac{1}{2}}^\mathsf{T}(u_k-x_{k+\frac{1}{2}})}{K}. 
\end{align*}
By invoking the definition of $G(x)$ and letting $D\triangleq 2D_T^2+\|x_0\|^2+\|u_0\|^2$, we obtain the following relation.
\begin{align}
\notag \gamma G(\bar{x}_{K}) & \le\tfrac{D^2}{K}+\tfrac{\gamma^2\sum_{k=0}^{K-1} (2\|\bar{w}_k-\bar{w}_{k+\frac{1}{2}}\|^2+\|\bar{w}_{k+\frac{1}{2}}\|^2)}{2K}\\&+ \tfrac{\gamma\sum_{k=0}^{K-1}  \bar{w}_{k+\frac{1}{2}}^\mathsf{T}(u_k-x_{k+\frac{1}{2}})}{K}. \label{eq-s4}
\end{align}
Before proceeding, we establish bounds for $x_k$ and $x_{k+\frac{1}{2}}$. From Proposition~\ref{ass}, we know $\{x_k\}$ converges to $x^*$ which indicates $\|x_k-x^*\|$ is bounded. We denote this bound by $\|x_k-x^*\| \le D_*$. By definition of $x_{k+\frac{1}{2}}$, it follow that 
\begin{align*}
&\|x_{k+\frac{1}{2}}-x^*\|=\left\|(\mathbf{I}+\gamma B)^{-1}(x_k - \gamma A(x_k)) \right. \\& \left.-(\mathbf{I}+\gamma B)^{-1}(x^* - \gamma A(x^*))\right\| \\
&\le \|(x_k - \gamma A(x_k))-(x^* - \gamma A(x^*))\| \\&\le (1+\gamma L)\|x_k-x^*\| \le (1+\gamma L)D_*,
\end{align*}
where the first inequality follows from that $(\mathbf{I}+\gamma B)^{-1}$ is a non-expansive operator.
Taking expectations on both sides of \eqref{eq-s4}, leads to the following inequality. 
\begin{align} \notag
 \mathbb{E}&[\gamma G(\bar{x}_{K})] \leq  {\tfrac{D^2}{K}}+\tfrac{\gamma^2\sum_{k=0}^{K-1} 2\mathbb{E}[\|\bar{w}_k-\bar{w}_{k+\frac{1}{2}}\|^2]+2\mathbb{E}[\|\bar{w}_{k+\frac{1}{2}}\|^2]}{2K} \notag 
\\ \notag & 
	 + \tfrac{\gamma \sum_{k=0}^{K-1} \mathbb{E}[\bar{w}_{k+\frac{1}{2}}^\mathsf{T}({u_k}-x_{k+\frac{1}{2}})]}{K}  \\\notag &\leq \tfrac{{2D^2}+{\gamma^2}\sum_{k=0}^{K-1} \tfrac{{\nu_1^2(4\|x_k\|^2+6\|x_{k+\frac{1}{2}}\|^2)+10\nu_2^2}}{N_k}}{2K}  
	\\ & \leq \tfrac{{2D^2}+{\gamma^2}\sum_{k=0}^{K-1} \tfrac{{\nu_1^2((8+12(1+\gamma L)^2)D_*^2+20\|x^*\|^2)+10\nu_2^2}}{N_k}}{2K} \label{eq1a}\\
	& \leq \tfrac{2D^2+\gamma^2 M ({{\nu_1^2((8+12(1+\gamma L)^2)D_*^2+20\|x^*\|^2)+10\nu_2^2}})}{2K}
{=\tfrac{\widehat{C}}{K}}, \notag
\end{align}
by defining $\widehat{C} \triangleq (2D^2+\\\gamma^2 M({{\nu_1^2((8+12(1+\gamma L)^2)D_*^2+20\|x^*\|^2)+10\nu_2^2}}))/2$. It follows that $\mathbb{E}[G(\bar{x}_{K})] \leq \widehat{C}/(\gamma K) = \mathcal{O}(1/K).$ \\
(b) For $\epsilon$ sufficiently small and when 
$\widetilde{C}$ is an appropriate constant, the result follows.
\begin{align*}
\sum_{k=0}^KN_k&\le\sum_{k=0}^{\lceil(\widehat{C}/\epsilon)\rceil}N_k\le\sum_{k=0}^{\lceil(\widehat{C}/\epsilon)\rceil}k^a\le\int_{k=0}^{(\widehat{C}/\epsilon)+1}x^a dx \\ &\le\tfrac{((\widehat{C}/\epsilon)+1)^{a+1}}{a+1}\le\left(\tfrac{\widetilde{C}}{\epsilon^{a+1}}\right).
\end{align*}
\end{proof}
\ic{
\noindent {\bf Comment.} A rate statement for the last iterate can also be
derived as well as shown in~\cite{iusem2017extragradient,bot2020mini}. Let
$K_\epsilon \triangleq \inf \{k \ge 1 : \mathbb{E}[r_\gamma^2(x_k)] \le
\epsilon\}$ where finiteness of $K_\epsilon$ can be shown a finite number,
allowing for showing that $\mathbb{E}[r^2_\gamma(x_{K_\epsilon})] \le
\mathcal{O}\left(\tfrac{1}{K_\epsilon}\right)$. Therefore for $K \geq
K_\epsilon$ iterations, we obtain a rate $\mathbb{E}[r^2_\gamma(x_{K})] $.
However, this avenue produces a local rate since we remain unclear regarding
the number of steps required to satisfy $\mathbb{E}[r^2_\gamma(x_{K})] \le
\epsilon$.}

\subsection{Convergence analysis under strongly monotone $A$}\label{sec:4.3}
In this subsection, we conduct an analysis under a strong monotonicity requirement. 
\begin{assumption} \label{smonotone}
The mapping $A$ is $\sigma$-strongly monotone, i.e.,
$(A(x)-A(y))^\mathsf{T}(x-y)\ge \sigma\|x-y\|^2,\quad\forall x,y \in \mathbb{R}^n.$ 
\end{assumption} 
The following lemma is essential to our rate of convergence analysis.
\begin{lemma}\label{inT}
Let Assumptions \ref{lip-mon} and \ref{smonotone} hold. Then the following holds for every $k$.
\begin{align} \label{smon-ineq} \notag
\notag\|x_{k+1}&-x^*\|^2 \leq (1-\sigma\gamma+\gamma^2)\|x_k-x^*\|^2 \\\notag&-(1-2\gamma^2(L^2+\tfrac{1}{2})-2\sigma\gamma)\|x_k-x_{k+\frac{1}{2}}\|^2\\
& +(4\gamma^2+2)\|\bar{w}_{k+\frac{1}{2}}\|^2 +4\gamma^2\|\bar{w}_k\|^2. 
\end{align}
\end{lemma}
\begin{proof}
According to Assumption \ref{smonotone}, we have
\begin{align}
\notag-2&\gamma(u_{k+\frac{1}{2}}+v_{k+\frac{1}{2}})^\mathsf{T}(x_{k+\frac{1}{2}}-x^*)\le-2\gamma\sigma\|x_{k+\frac{1}{2}}-x^*\|^2\\&\le2\gamma\sigma\|x_{k+\frac{1}{2}}-x_k\|^2-\gamma\sigma\|x_k-x^*\|^2. \label{eq56}
\end{align}
Using \eqref{eq56} in \eqref{new21}, we deduce
\begin{align}
\notag&\|x_{k+1}-x^*\|^2\le \|x_k-x^*\|^2-(1-2\gamma^2L^2)\|x_k-x_{k+\frac{1}{2}}\|^2\\\notag&-2\gamma\sigma\|x_{k+\frac{1}{2}}-x^*\|^2+2\gamma^2\|\bar{w}_{k+\frac{1}{2}}-\bar{w}_k\|^2\\\notag&-2\gamma \bar{w}_{k+\frac{1}{2}}^\mathsf{T}(x_{k+\frac{1}{2}}-x^*) \\
\notag &\le \|x_k-x^*\|^2-(1-2\gamma^2L^2)\|x_k-x_{k+\frac{1}{2}}\|^2\\\notag&+2\gamma\sigma\|x_{k+\frac{1}{2}}-x_k\|^2-\gamma\sigma\|x_k-x^*\|^2+2\gamma^2\|\bar{w}_{k+\frac{1}{2}}-\bar{w}_k\|^2\\\notag&-2\gamma \bar{w}_{k+\frac{1}{2}}^\mathsf{T}(x_{k+\frac{1}{2}}-x^*) \\
\notag &\le (1-\sigma\gamma)\|x_k-x^*\|^2-(1-2\gamma^2L^2-2\sigma\gamma)\|x_k-x_{k+\frac{1}{2}}\|^2\\\notag&+2\gamma^2\|\bar{w}_{k+\frac{1}{2}}-\bar{w}_k\|^2-2\gamma \bar{w}_{k+\frac{1}{2}}^\mathsf{T}(x_{k+\frac{1}{2}}-x^*) \\
\notag&\le(1-\sigma\gamma)\|x_k-x^*\|^2-(1-2\gamma^2L^2-2\sigma\gamma)\|x_k-x_{k+\frac{1}{2}}\|^2\\\notag&+4\gamma^2\|\bar{w}_{k+\frac{1}{2}}\|^2+4\gamma^2\|\bar{w}_k\|^2-2\gamma \bar{w}_{k+\frac{1}{2}}^\mathsf{T}(x_{k+\frac{1}{2}}-x^*) \\
\notag& = (1-\sigma\gamma)\|x_k-x^*\|^2-(1-2\gamma^2L^2-2\sigma\gamma)\|x_k-x_{k+\frac{1}{2}}\|^2\\\notag&+4\gamma^2\|\bar{w}_{k+\frac{1}{2}}\|^2+4\gamma^2\|\bar{w}_k\|^2-2\gamma \bar{w}_{k+\frac{1}{2}}^\mathsf{T}(x_{k+\frac{1}{2}}-x_k)\\\notag&-2\gamma \bar{w}_{k+\frac{1}{2}}^\mathsf{T}(x_{k}-x^*) \\ 
\notag& \leq (1-\sigma\gamma+\gamma^2)\|x_k-x^*\|^2+(4\gamma^2+2)\|\bar{w}_{k+\frac{1}{2}}\|^2 \\
\notag& -(1-2\gamma^2(L^2+\tfrac{1}{2})-2\sigma\gamma)\|x_k-x_{k+\frac{1}{2}}\|^2
\notag+4\gamma^2\|\bar{w}_k\|^2.
\end{align}
\end{proof}

\begin{theorem}[{\bf a.s. convergence without unbiasedness}]
Let Assumptions \ref{lip-mon}, \ref{moment}(ii) and \ref{smonotone} hold.
Consider a sequence $\{x_k\}$ generated by ({\bf vr-SMBFS}). Suppose $N_0 \geq
\tfrac{2(24 \gamma^2+8)\nu_1^2}{\sigma \gamma}$, $\gamma < \min\left\{ \tfrac{\sigma}{4}, \tfrac{1}{20\sigma}, \tfrac{\sqrt{7}}{4\tilde L}\right\}$, $\{N_k\}$ is a
non-decreasing sequence, and
$\sum_{k} \tfrac{1}{N_k} < \infty$, ${\tilde L}^2= L^2 + \tfrac{1}{2}$.
Then $\{x_k\}$ converges to $x^*$ in an a.s. sense. 
\end{theorem}
\begin{proof}
By taking conditional expectations on both sides of \eqref{smon-ineq}, we obtain the following relation by invoking Assumption~\ref{moment}(ii) and defining {$\tilde{L}^2 = L^2+\tfrac{1}{2}$.}  
\begin{align*}
\mathbb{E}&[\|x_{k+1}-x^*\|^2 \mid \mathcal F_k]  \leq (1-\sigma\gamma+\gamma^2)\|x_k-x^*\|^2\\&-(1-2\gamma^2{\tilde L}^2-2\sigma\gamma)\|x_k-x_{k+\frac{1}{2}}\|^2\\
& +(4\gamma^2+2)\mathbb{E}[\|\bar{w}_{k+\frac{1}{2}}\|^2 \mid \mathcal{F}_k] +4\gamma^2\mathbb{E}[\|\bar{w}_k\|^2 \mid \mathcal{F}_k] \\
& \leq (1-\sigma\gamma+\gamma^2)\|x_k-x^*\|^2\\&-(1-2\gamma^2{\tilde L}^2-2\sigma\gamma)\|x_k-x_{k+\frac{1}{2}}\|^2+4\gamma^2\mathbb{E}[\|\bar{w}_k\|^2 \mid \mathcal{F}_k]\\
& +(4\gamma^2+2)\mathbb{E}[\mathbb{E}[\|\bar{w}_{k+\frac{1}{2}}\|^2\mid \mathcal{F}_{k+\frac{1}{2}}] \mid \mathcal{F}_k]  \\
& \leq (1-\sigma\gamma+\gamma^2)\|x_k-x^*\|^2\\&-(1-2\gamma^2{\tilde L}^2-2\sigma\gamma)\|x_k-x_{k+\frac{1}{2}}\|^2\\
& +\tfrac{(4\gamma^2+2)(\nu_1^2\|x_{k+\frac{1}{2}}\|^2 + \nu_2^2)}{N_{k}}+\tfrac{4\gamma^2(\nu_1^2 \|x_k\|^2 + \nu_2^2)}{N_k}.
\end{align*}
We now derive bounds on the last two terms, leading to the following inequality. 
\begin{align}
\notag\mathbb{E}&[\|x_{k+1}-x^*\|^2 \mid \mathcal F_k]   \leq (1-\sigma\gamma+\gamma^2)\|x_k-x^*\|^2\\\notag&-(1-2\gamma^2{\tilde L}^2-2\sigma\gamma)\|x_k-x_{k+\frac{1}{2}}\|^2\\ \notag
& +\tfrac{(4\gamma^2+2)(\nu_1^2\|x_{k+\frac{1}{2}}\|^2 + \nu_2^2)}{N_{k}}+\tfrac{4\gamma^2(\nu_1^2 \|x_k\|^2 + \nu_2^2)}{N_k}\\
	&\notag \leq (1-\sigma\gamma+\gamma^2)\|x_k-x^*\|^2\\\notag&-(1-2\gamma^2{\tilde L}^2-2\sigma\gamma)\|x_k-x_{k+\frac{1}{2}}\|^2\\ \notag
& +\tfrac{(4\gamma^2+2)(2\nu_1^2\|x_{k+\frac{1}{2}}-x_k\|^2)}{N_{k}}+\tfrac{(12\gamma^2+4)\nu_1^2 \|x_k\|^2 + (12\gamma^2 + 4) \nu_2^2)}{N_k}\\
	&\notag \leq (1-\sigma\gamma+\gamma^2)\|x_k-x^*\|^2\\\notag&-(1-2\gamma^2{\tilde L}^2-2\sigma\gamma)\|x_k-x_{k+\frac{1}{2}}\|^2\\ \notag
& +\tfrac{(8\gamma^2+4)(\nu_1^2\|x_{k+\frac{1}{2}}-x_k\|^2)}{N_{k}}\\\notag&+\tfrac{(24\gamma^2+8)\nu_1^2 (\|x_k-x^*\|^2 + \|x^*\|^2) + (12\gamma^2 + 4) \nu_2^2)}{N_k}\\
	&\notag \leq \left(1-\sigma\gamma+\gamma^2+\tfrac{(24\gamma^2+8)\nu_1^2}{N_0}\right)\|x_k-x^*\|^2\\ \notag
	& -\left(1-2\gamma^2{\tilde L}^2-2\sigma\gamma-\tfrac{(8\gamma^2+4)\nu_1^2}{N_{0}}\right)\|x_k-x_{k+\frac{1}{2}}\|^2+\delta_k \\ \notag
&\notag \leq \left(1-\sigma\gamma+\gamma^2+\tfrac{(24\gamma^2+8)\nu_1^2}{N_0}\right)\|x_k-x^*\|^2\\ \notag
	& -\left(1-2\gamma^2\tilde L^2-2\sigma\gamma-\tfrac{(24\gamma^2+8)\nu_1^2}{N_0}\right)\|x_k-x_{k+\frac{1}{2}}\|^2+\delta_k\\ \notag
& \leq \left(1-\tfrac{1}{2}\sigma\gamma+\gamma^2\right)\|x_k-x^*\|^2 \\&-\left(1-2\gamma^2\tilde{L}^2-\tfrac{5}{2}\sigma\gamma\right)\|x_k-x_{k+\frac{1}{2}}\|^2 + \delta_k, \label{ineq-ce}
\end{align}
where $\delta_k \triangleq \tfrac{(24\gamma^2+8)\nu_1^2 \|x^*\|^2 + (12\gamma^2 + 4) \nu_2^2}{N_k}$ and the final inequality follows from $N_0 \geq \tfrac{2(24 \gamma^2+8)\nu_1^2}{\sigma \gamma}$. We observe that 
\begin{align*}
 &\left(1- \tfrac{1}{2}\sigma\gamma+\gamma^2\right) = \left(1-\gamma(\tfrac{\sigma}{2} - \gamma)\right) \overset{\tiny \gamma < \frac{\sigma}{4}}{ < } \left( 1-\tfrac{\gamma\sigma}{4} \right) < 1\\
  &\left(1- \tfrac{1}{2}\sigma\gamma+\gamma^2\right) \overset{\tiny \gamma < \frac{2}{\sigma}}{ > } (1-1+\gamma^2) >0\\
 &{-}\left(1-2\gamma^2\tilde{L}^2-\tfrac{5}{2}\sigma\gamma\right) \overset{\tiny \gamma < \frac{1}{20\sigma}}{<} -\left(\tfrac{7}{8}-2\gamma^2\tilde{L}^2\right)  \overset{\tiny \gamma^2 < \frac{7}{16\tilde{L}^2}}{ < } 0.  
\end{align*}
In other words, if $\gamma < \min\left\{ \tfrac{\sigma}{4}, \tfrac{1}{20\sigma}, \tfrac{\sqrt{7}}{4\tilde L}\right\}$, \eqref{ineq-ce} can be further bounded by $(1-a_k)\|x_k-x^*\|^2 + \delta_k$ where $a_k = \tfrac{\gamma \sigma}{4}$ for all $k$ and  $a_k,\delta_k$ satisfy Lemma~\ref{as-recur}. Consequently,  $\|x_k-x^*\|^2 \to 0$ in an a.s. sense as $k \to \infty$.  

\end{proof}
Next we provide rate and complexity statements involving ({\bf vr-SMFBS}) under geometrically increasing $N_k$.
\begin{proposition}[{\bf Linear convergence}] \label{rateTvs} 
Let Assumptions~\ref{lip-mon}, \ref{moment}(ii) and \ref{smonotone} hold. Consider a sequence $\{x_k\}$ generated by ({\bf vr-SMBFS}). Suppose $N_0 \geq
\tfrac{2(24 \gamma^2+8)\nu_1^2}{\sigma \gamma}$, $\gamma < \min\left\{ \tfrac{\sigma}{4}, \tfrac{1}{20\sigma}, \tfrac{\sqrt{7}}{4\tilde L}\right\},
\sum_{k} \tfrac{1}{N_k} < \infty$, ${\tilde L}^2= L^2 + \tfrac{1}{2}$, $\|x^0-x^*\| \leq D_0$ and $N_k \triangleq N_0\lfloor \rho^{-(k+1)} \rfloor$ for all $k > 0$.
Then the following hold.

\noindent (a) Suppose $q\triangleq (1-\tfrac{\sigma\gamma}{4}) < 1$. Then $\mathbb{E}[\|x_k-x^*\|^2] \leq \tilde{D} \tilde{\rho}^k$ where $\tilde D > 0$  and $\tilde \rho = \max\{q,\rho\}$ if $q \neq \rho$ and $\tilde \rho \in (q,1)$ if $q = \rho$.

\noindent (b) Suppose $x_{K+1}$ is such that $\mathbb{E}[\|x_{K+1}-x^*\|^2] \le \epsilon$. Then the oracle complexity is $\sum_{k=0}^KN_k \le \mathcal{O}\left(\frac{1}{\epsilon}\right).$
\end{proposition}
\begin{proof}
(a) 
By taking unconditional expectations on both sides of \eqref{ineq-ce}, we obtain
\begin{align}
\mathbb{E}[\|x_{k+1}-x^*\|^2] & \le q\mathbb{E}[\|x_k-x^*\|^2]+\tfrac{D}{N_k}, \label{eq58v}
\end{align}
where $D\triangleq (24\gamma^2+8)\nu_1^2 \|x^*\|^2 + (12\gamma^2 + 4) \nu_2^2$ and $q\triangleq (1-\tfrac{\sigma\gamma}{4})$.
Recall that $N_k$ can be bounded as seen next.
\begin{align}
N_k= N_0\lfloor \rho^{-(k+1)} \rfloor \ge N_0\left\lceil \tfrac{1}{2}\rho^{-(k+1)} \right\rceil \ge \tfrac{N_0}{2}\rho^{-(k+1)}. \label{eq59p-t}
\end{align}
We now consider three cases.

\noindent  (i): $q<\rho<1$. Using \eqref{eq59p-t} in \eqref{eq58v} and defining $\bar D \triangleq \tfrac{2D}{N_0}$ and $\tilde D \triangleq D_0+\bar D$, we obtain
\begin{align}
\notag\mathbb{E}[\|&x_{k+1}-x^*\|^2] \le q\mathbb{E}[\|x_k-x^*\|^2]+\tfrac{D}{N_{k}} \\\notag&\leq q\mathbb{E}[\|x_k-x^*\|^2]+\tfrac{2D}{N_0}\rho^{k+1} \\
\notag&\le q^{k+1}\|x_0-x^*\|+\bar D\sum_{j=1}^{k+1}q^{k+1-j}\rho^j 
\\\notag&\le D_0q^{k+1}+\bar D\rho^{k+1}\sum_{j=1}^{k+1}(\tfrac{q}{\rho})^{k+1-j}\le \tilde{D}\rho^{k+1}.
\end{align}

\noindent  (ii): $\rho<q<1$. Akin to (i) and defining $\tilde D$ apprioriately,  $\mathbb{E}[\|x_{k+1}-x^*\|^2] \le\tilde{D}q^{k+1}$. \\

\noindent (iii): $\rho=q<1$. 
If $\tilde{\rho} \in (q,1)$ and $\widehat{D} > \tfrac{1}{\ln(\tilde{\rho}/q)^e}$, proceeding similarly we obtain
\begin{align}
\notag\mathbb{E}[&\|x_{k+1}-x^*\|^2] \le q^{k+1}\mathbb{E}[\|x_0-x^*\|^2]+\bar D\sum_{j=1}^{k+1}q^{k+1} \\\notag&\le D_0q^{k+1}+\bar D\sum_{j=1}^{k+1}q^{k+1}= D_0q^{k+1}+\bar D(k+1)q^{k+1} \\\notag
&\overset{\tiny \mbox{\cite[Lemma~4]{ahmadi2016analysis}}}{\le}  \tilde{D} \tilde{\rho}^{k+1}, \mbox{ where } \tilde{D} \triangleq (D_0+\bar D \cdot \widehat{D}). 
\end{align}
\noindent Thus, $\{x_k\}$ converges linearly in an expected-value sense. \\ 
(b) Case (i): If $q<\rho<1$. From (a), it follows that
\begin{align*}
\mathbb{E}[\|x_{K+1}-x^*\|^2] &\le \tilde{D}\rho^{K+1}\leq \  \epsilon \Longrightarrow  K \ge  \log_{1/\rho}(\tilde{D}/\epsilon) - 1 .
\end{align*} 
If $K = \lceil \log_{1/\rho}(\tilde{D}/\epsilon)\rceil - 1$, then ({\bf vr-SMFBS}) requires $\sum_{k=0}^KN_k$ evaluations. Since $N_k=N_0\lfloor \rho^{-(k+1)} \rfloor \le N_0\rho^{-(k+1)}$, then we have 
\begin {align*}
 \quad \sum_{k=0}^{\lceil \log_{1/\rho }(\tilde{D}/\epsilon)\rceil -1}&N_0\rho^{-(k+1)}  =
\sum_{t=1}^{\lceil \log_{1/\rho }(\tilde{D}/\epsilon)\rceil}N_0\rho^{-t} \\
& \le \tfrac{N_0}{\rho^2\left(\tfrac{1}{\rho}-1\right)}\left(\tfrac{1}{\rho}\right)^{\lceil \log_{1/\rho}(\tilde{D}/\epsilon)\rceil}\\ 
& \le
\tfrac{N_0}{\rho\left(\tfrac{1}{\rho}-1\right)}\left(\tfrac{1}{\rho}\right)^{
\log_{1/\rho}(\tilde{D}/\epsilon)+1} \\
 &\le \tfrac{N_0}{\left(1-\rho\right)}\left(\tfrac{1}{\rho}\right)^{\log_{1/\rho}(\tilde{D}/\epsilon)}
  \le \tfrac{N_0}{(1-\rho)}\left(\tfrac{\tilde{D}}{\epsilon}\right).
\end{align*}
We omit cases (ii) and (iii) which lead to similar complexities.
\end{proof}

\noindent {\bf Remark.} We comment on our findings next.\\ 
\noindent (a)  {\em Rates and asymptotics.} We believe that the findings fill important gaps in terms of providing rate statements for monotone inclusions. In particular, the rate statement in monotone regimes relies on utilizing a lesser known gap function while the variance-reduced schemes achieve deterministic rates of convergence. In addition, the oracle complexities are near-optimal.\\ 
\noindent (b) {\em Algorithm parameters.} Akin to more traditional first-order schemes, these schemes rely on utilizing constant steplengths and leverage problem parameters such as Lipschitz and strong monotonicity constants. We believe that by using diminishing steplength sequences, we may be able to derive weaker rate statements that do not rely on problem parameters.\\
\noindent (c) {\em Expectation-valued $B$.} We may consider a setting where $B$ is expectation-valued and the resolvent operation is approximated via stochastic approximation. 

\section{Numerical Results}\label{sec:numerics}
{In this section, we apply the proposed schemes on a 
2-stage SVI problem described in Section~\ref{sec:app} (Example b).} \\
\noindent {\em Problem parameters for 2-stage SVI.} We
generate a set of $J$ i.i.d samples $\{\xi_i\}^J_{i=1}$,
where $\xi_i \sim U[-5,0]$. Suppose $h_i(\omega)=\xi_i$ {for
$i=1,\cdots, J$}. In
addition,
$c_i(x_i)=\tfrac{1}{2}m_ix_i^2+\ell_i{x_i}$, $M\in\mathbb{R}^{J\times J}$ is a diagonal matrix with
nonnegative elements
$M = \mbox{diag}(m_1,\dots,m_J)$ while
$\ell=[\ell_1,\dots,\ell_J]^\mathsf{T}\in\mathbb{R}^J$ where $\ell_i \in U(2,3)$. Furthermore, the inverse demand function
$p$ is defined as $p(X) = d - rX$ where $d = 1$ and $r = 1$.
Thus $G(x)$, as
defined in Section 1.1 (b), can be simplified as $G(x)=Mx+\ell$. In this setting, $A(x)=Mx+\ell+R(x)+D^\epsilon(x)$ and $B(x)=\mathcal{N}_\mathcal{X}(x)$, where $\mathcal{X} \triangleq \mathbb{R}_J^+$. The Lipschitz of $A$ is given by $L = L_B + L_R + L_D^{\epsilon}$ where  $L_B = \max_{i} m_i$, $L_R = r\| \mathbf{I}+ {\bf 1}{\bf 1}^\mathsf{T}\|$ and $L_D^{\epsilon} = \tfrac{1}{\epsilon}$. All the schemes are implemented in MATLAB on a PC with 16GB RAM and 6-Core Intel Core i7 processor
(2.6GHz).

\begin{figure}[htbp]
\begin{center}
\includegraphics[width=.4\textwidth]{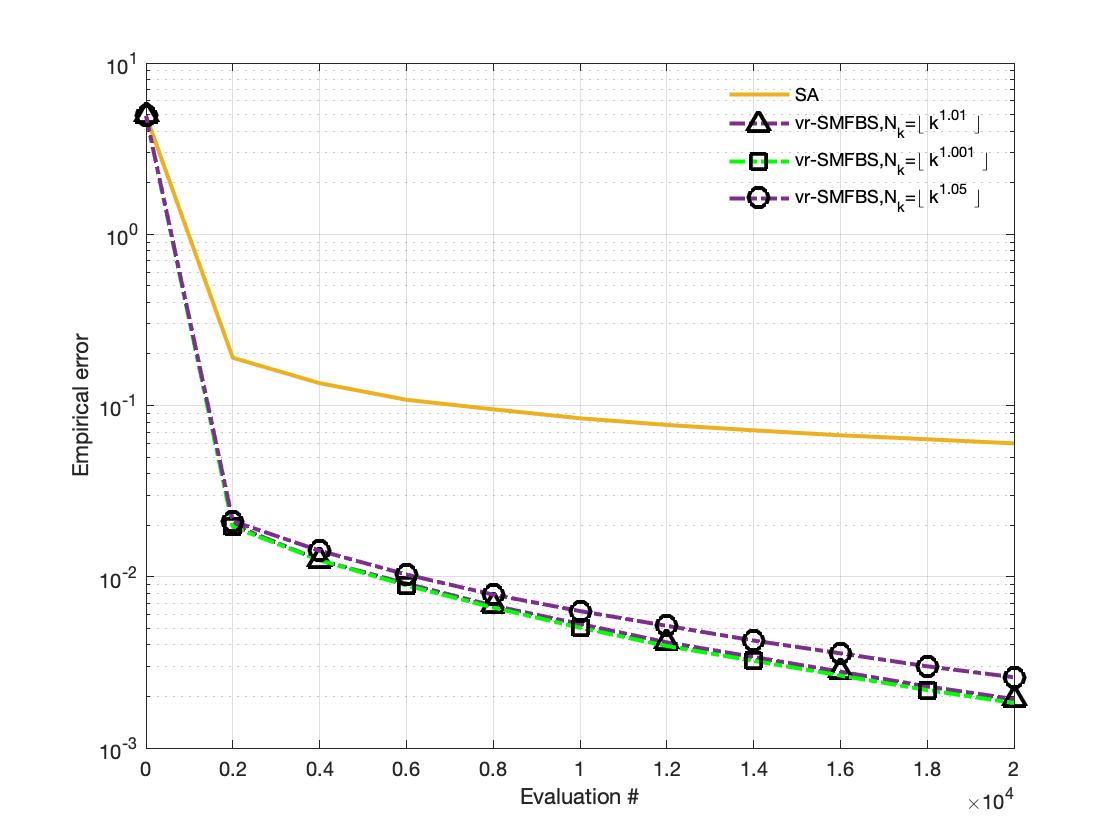}
\includegraphics[width=.4\textwidth]{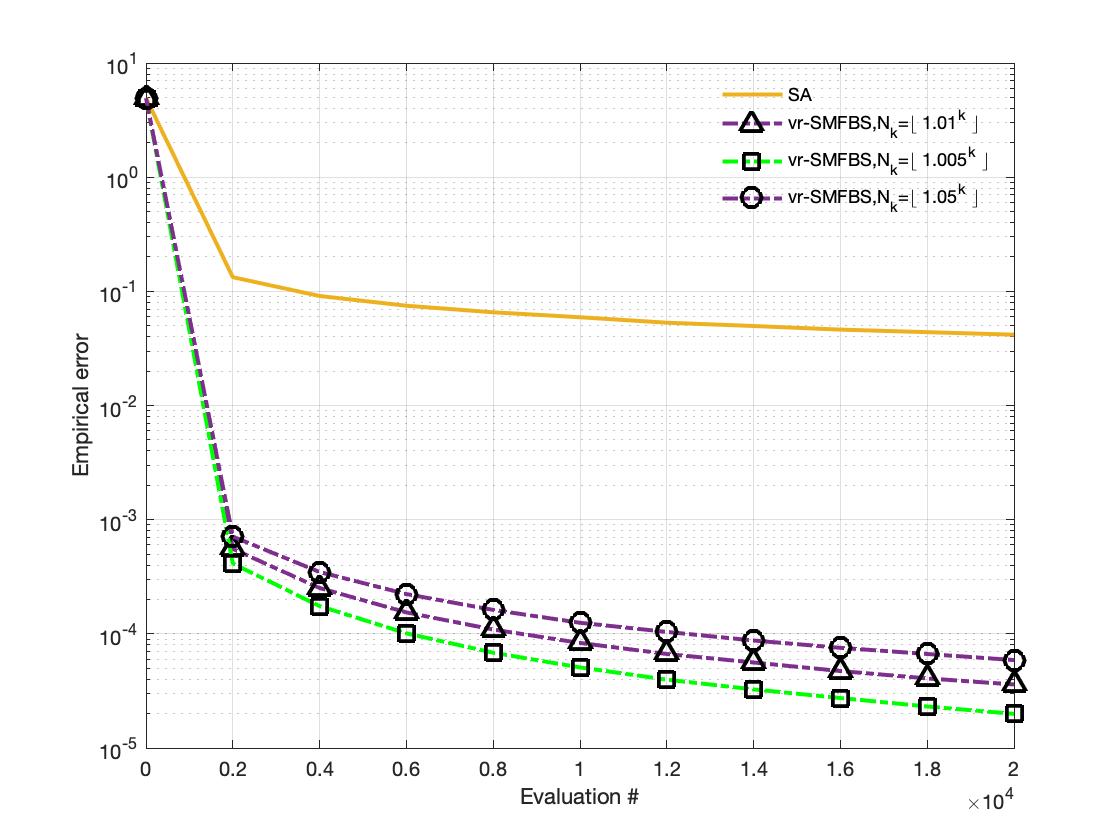}
\end{center}
\caption{Trajectories for (SA) and (vr-SMBFS)  (up: monotone; down: s-monotone)}
\label{pro1}
\end{figure}

We describe the three schemes being compared and specify their algorithm parameters. Solution quality is compared by estimating the residual function  $\texttt{res}(x)=\|x-\Pi_\mathcal{X}(x-\gamma A(x))\|$.

\subsection{Algorithm specifications.} 

\noindent {(i) {\bf (SA)}: Stochastic approximation scheme.}  {The {\bf (SA)} scheme utilizes the following update.
\begin{align}\tag{SA} x_{k+1}:= \Pi_{\mathcal{X}} \left[x_k - \gamma_k A(x_k,\omega_k)\right],\end{align}
where $A(x_k) = \mathbb{E}[A(x_k,\omega_k)]$, and $\gamma_k \triangleq \tfrac{1}{\sqrt{k}}$. $x_0$ is randomly generated in $[0,1]^J$.

\noindent {(ii) {\bf (vr-SMFBS)}: Variance-reduction stochastic modified forward-backward scheme.} We choose a constant $\gamma=\tfrac{1}{4L}$ which satisfies the steplength assumption and we assume $N_k=\lfloor k^{1.01}\rfloor$ for merely monotone problems, $N_k=\lfloor 1.01^{k+1}\rfloor$ for strongly monotone problems.

\subsection{Performance comparison and insights.}
In Fig. \ref{pro1}, we compare both schemes under mere monotonicity and
strong monotonicity, respectively and examine sensitivities to
the sample growth rate. Standard SA schemes may struggle when the
problem is ill-conditioned and we examine the performance of the schemes in
such regimes and provide the results in  for merely monotone and strongly monotone settings in Tables
\ref{cpb1} and \ref{cpb2}, respectively.

\begin{table}[htbp]
\scriptsize
\caption{Comparison of ({\bf vr-SMFBS}) with (SA)}
\begin{center}
    \begin{tabular}[t]{ c | c | c | c | c | c | c }
    \hline
    \multicolumn{7}{c} {merely monotone, 20000 evaluations} \\ \hline
    \multirow{2}{*}{$L$} & \multicolumn{3}{c|} {vr-SMFBS} & \multicolumn{3}{c} {SA} \\ \cline{2-7}
      & error & time & CI & error & time & CI  \\ \hline
      1e1 & 1.6e-3 & 2.6 & [1.3e-3,1.8e-3] & 5.3e-2 & 2.7 & [5.0e-2,5.7e-2]  \\ \hline
      1e2 & 1.9e-3 & 2.6 & [1.6e-3,2.1e-3] & 6.1e-2 & 2.7 & [5.8e-2,6.4e-2]  \\ 
 \hline
      1e3 & 2.2e-3 & 2.6 & [2.0e-3,2.5e-3] & 7.6e-2 & 2.5 & [7.3e-2,7.9e-2] \\ 
 \hline
 1e4 & 5.9e-3 & 2.6 & [5.4e-3,6.2e-3] & 9.4e-2 & 2.6 & [9.0e-1,9.7e-1] \\ 
 \hline
  \hline
  \multicolumn{7}{c} {Complicated $\mathcal X$, merely monotone, 2000 evaluations} \\ \cline{1-7}
 1e2 & 1.9e-3 & 6.8 & [1.6e-3,2.0e-3] & 6.0e-2 & 232 & [5.7e-2,6.3e-2] \\ 
 \cline{1-7}
      \end{tabular}
\end{center}
\label{cpb1}
\end{table}

\begin{table}[!htb]
\scriptsize
\caption{Comparison of ({\bf vr-SMFBS}) with (SA)}
\begin{center}
    \begin{tabular}[t]{ c | c | c | c | c | c | c }
    \hline
    \multicolumn{7}{c} {strongly monotone, 20000 evaluations} \\ \hline
    \multirow{2}{*}{$L$} & \multicolumn{3}{c|} {vr-SMFBS} & \multicolumn{3}{c} {SA} \\ \cline{2-7}
      & error & time & CI & error & time & CI  \\ \hline
      1e1 & 1.5e-5 & 2.6 & [1.2e-5,1.7e-5] & 2.9e-2 & 2.5 & [2.7e-2,3.1e-2]  \\ \hline
      1e2 & 3.6e-5 & 2.5 & [3.3e-5,3.9e-5] & 4.1e-2 & 2.5 & [3.8e-2,4.4e-2]  \\ 
 \hline
      1e3 & 5.6e-5 & 2.5 & [4.2e-6,4.7e-6] & 5.5e-2 & 2.4 & [5.2e-2,5.7e-2]  \\ 
 \hline
      1e4 & 7.4e-5 & 2.5 & [7.1e-5,7.7e-5] & 6.0e-2 & 2.5 & [5.7e-2,6.3e-2]  \\ 
 \hline
  \hline
  \multicolumn{7}{c} {Complicated $\mathcal X$, strongly monotone, 2000 evaluations} \\ \cline{1-7}
 1e3 & 5.6e-5 & 18 & [4.2e-6,4.7e-6] & 5.5e-2 & 234 & [5.2e-2,5.8e-2] \\ 
 \cline{1-7}
      \end{tabular}
\end{center}
\label{cpb2}
\end{table}

\noindent {\bf Key findings.} ({\bf vr-SMFBS}) trajectories are characterized by significantly smaller empirical errors than
({\bf SA}). There is little impact on ({\bf vr-SMFBS}) when
varying the sample growth rate. Moreover, ({\bf vr-SMFBS}) appears to cope better with
large Lipchitz constant. Since we utilize the analytical $L$ to set
$\gamma$, we see smaller steps for large $L$. To show the efficiency
of ({\bf vr-SMFBS}) with complicated feasible set, we change
$\mathcal{X} \triangleq \{x \in \mathbb{R}_J^+ \mid \sum_i x_i \le 10\}$
which  leads to computationally expensive projection steps. As
seen in  the last row of Tables \ref{cpb1} and \ref{cpb2}, ({\bf
vr-SMFBS}) takes far less time than ({\bf SA}).

\subsection{Comparison with SAA schemes}
To show the performance of our proposed schemes , we consider the ({\bf SAA}) scheme used in \cite{jiang2019regularized}. Let $(\omega_{1i})_{i=1}^J,(\omega_{2i})_{i=1}^J$, $\dots,(\omega_{\nu i})_{i=1}^J$ denote independent identically distributed (i.i.d.) samples. Then, with ({\bf SAA}) we solve the following formulation of problem:
\begin{align*}
\left\{\begin{aligned}
 0 &\leq x_i \perp  c_i'(x_i) + r \cdot (X+x_i) - d + \tfrac{1}{\nu}\sum_{l=1}^\nu\lambda_i(\omega_{li}) \geq 0 \\
 0 &\leq y_i(\omega_{li})  \perp h_i(\omega_{li}) + \lambda_i(\omega_{li}) \geq 0\\
0 &\leq \lambda_i(\omega_{li})  \perp x_i - y_i(\omega_{li}) + \epsilon\lambda_i(\omega_{li}) \geq 0, \ \forall l=1,\dots,\nu
\end{aligned}\right\} \\ \forall i=1,\dots,J.
\end{align*} 
This problem is cast as a linear complementarity problem (LCP), allowing for utilizing \texttt{PATH}~\cite{dirkse1995path} to compute a solution.
We compare ({\bf SAA}) with ({\bf vr-SMFBS}) in Table \ref{cp1}. From the results, we observe that although the empirical errors of both schemes are similar, the ({\bf SAA}) scheme takes far longer than ({\bf vr-SMFBS}) when using a large number of samples. In fact, ({\bf vr-SMFBS}) scale well with overall number of evaluations.
\begin{table}[htbp]
\scriptsize
\caption{Comparison of (SAA) with (vr-SMFBS) (u: s-monotone, d: monotone)}
\begin{center}
    \begin{tabular}{ c|cc|cc }
    \hline
    \multirow{2}{*}{$\nu$} & \multicolumn{2}{c|}{SAA} & \multicolumn{2}{c}{vr-SMFBS} \\\cline{2-5}
    & time/s & res & time/s & res  \\\hline
    1000 & 0.7 & 4.5e-4 & 0.3 & 4.8e-4  \\
    2000 & 3.3 & 3.5e-4 & 0.5 & 2.3e-4  \\
     4000 & 6.2 & 1.6e-4 & 0.6 & 1.0e-4  \\
    10000 & 32.7 & 3.7e-5 & 1.2 & 3.4e-5  \\
    20000 & 117.7 & 2.8e-5 & 2.5 & 1.5e-5  \\
    \hline
      \end{tabular} \bigskip \\ 
      
      \begin{tabular}{ c|cc|cc }
    \hline
    \multirow{2}{*}{$\nu$} & \multicolumn{2}{c|}{SAA} & \multicolumn{2}{c}{vr-SMFBS}  \\\cline{2-5}
    & time/s & res & time/s & res  \\\hline
    1000 & 0.7 & 5.6e-2 & 0.3 & 2.7e-2  \\
    2000 & 3.0 & 3.4e-2 & 0.5 & 2.0e-2  \\
     4000 & 5.8 & 2.2e-2 & 0.6 & 1.2e-2  \\
    10000 & 61.8 & 7.8e-3 & 1.2 & 5.3e-3  \\
    20000 & 115.0 & 2.5e-3 & 2.6 & 1.9e-3  \\
     \hline
      \end{tabular}
\end{center}
\label{cp1}
\end{table}

\section{Concluding remarks} Monotone inclusions represent an important class
of problems and their stochastic counterpart subsumes a large class of
stochastic optimization and equilibrium problems. Such objects arise in
optimization, game-theoretic, and model-predictive control problems afflicted
by uncertainty. We propose a variance-reduced splitting framework for resolving
such problems when the map is structured. Under suitable assumptions on the
sample-size, we prove that the scheme displays a.s. convergence guarantees and
achieves optimal linear and sublinear rates in strongly monotone and monotone
regimes while achieving either optimal or near-optimal sample-complexities. By
incorporating state-dependent bounds on noise and weakening unbiasedness
requirements (in strongly monotone settubfs), we develop techniques that can
accommodate far more general settings. Preliminary numerics on a class of
two-stage stochastic variational inequality problems suggest that the scheme
outperform stochastic approximation schemes, as well as sample-average
approximation approaches. 

\appendices
\bibliographystyle{ieeetr}
\bibliography{extrasvi}

\begin{thebibliography}{10}

\bibitem{robinson83generalized}
S.~M. Robinson, {\em Generalized Equations}, pp.~346--367.
\newblock Berlin, Heidelberg: Springer Berlin Heidelberg, 1983.

\bibitem{rockafellar1976monotone}
R.~T. Rockafellar, ``Monotone operators and the proximal point algorithm,''
  {\em SIAM Journal on Control and Optimization}, vol.~14, no.~5, pp.~877--898,
  1976.

\bibitem{glowinski1989augmented}
R.~Glowinski and P.~Le~Tallec, {\em Augmented Lagrangian and operator-splitting
  methods in nonlinear mechanics}, vol.~9.
\newblock SIAM, 1989.

\bibitem{douglas1956numerical}
J.~Douglas and H.~H. Rachford, ``On the numerical solution of heat conduction
  problems in two and three space variables,'' {\em Transactions of the
  American mathematical Society}, vol.~82, no.~2, pp.~421--439, 1956.

\bibitem{peaceman1955numerical}
D.~W. Peaceman and H.~H. Rachford, Jr, ``The numerical solution of parabolic
  and elliptic differential equations,'' {\em Journal of the Society for
  industrial and Applied Mathematics}, vol.~3, no.~1, pp.~28--41, 1955.

\bibitem{lions1979splitting}
P.-L. Lions and B.~Mercier, ``Splitting algorithms for the sum of two nonlinear
  operators,'' {\em SIAM Journal on Numerical Analysis}, vol.~16, no.~6,
  pp.~964--979, 1979.

\bibitem{passty1979ergodic}
G.~B. Passty, ``Ergodic convergence to a zero of the sum of monotone operators
  in {H}ilbert space,'' {\em Journal of Mathematical Analysis and
  Applications}, vol.~72, no.~2, pp.~383--390, 1979.

\bibitem{chen19convergence}
X.~Chen, A.~Shapiro, and H.~Sun, ``Convergence analysis of sample average
  approximation of two-stage stochastic generalized equations,'' {\em SIAM
  Journal on Optimization}, vol.~29, no.~1, pp.~135--161, 2019.

\bibitem{aumann1965integrals}
R.~J. Aumann, ``Integrals of set-valued functions,'' {\em Journal of
  Mathematical Analysis and Applications}, vol.~12, no.~1, pp.~1--12, 1965.

\bibitem{dantzig2010linear}
G.~B. Dantzig, ``Linear programming under uncertainty,'' in {\em Stochastic
  programming}, pp.~1--11, Springer, 2010.

\bibitem{birge2011introduction}
J.~R. Birge and F.~Louveaux, {\em Introduction to stochastic programming}.
\newblock Springer Science \& Business Media, 2011.

\bibitem{shapiro2014lectures}
A.~Shapiro, D.~Dentcheva, and A.~Ruszczy{\'n}ski, {\em Lectures on stochastic
  programming: modeling and theory}.
\newblock SIAM, 2014.

\bibitem{jiang08stochastic}
H.~Jiang and H.~Xu, ``Stochastic approximation approaches to the stochastic
  variational inequality problem,'' {\em IEEE Transactions on Automatic
  Control}, vol.~53, no.~6, pp.~1462--1475, 2008.

\bibitem{juditsky2011solving}
A.~Juditsky, A.~Nemirovski, and C.~Tauvel, ``Solving variational inequalities
  with stochastic mirror-prox algorithm,'' {\em Stochastic Systems}, vol.~1,
  no.~1, pp.~17--58, 2011.

\bibitem{shanbhag2013stochastic}
U.~V. Shanbhag, ``Stochastic variational inequality problems: Applications,
  analysis, and algorithms,'' in {\em Theory Driven by Influential
  Applications}, pp.~71--107, INFORMS, 2013.

\bibitem{ravat17existence}
U.~Ravat and U.~V. Shanbhag, ``On the existence of solutions to stochastic
  quasi-variational inequality and complementarity problems,'' {\em
  Mathematical Programming}, vol.~165, no.~1, pp.~291--330, 2017.

\bibitem{ravat11characterization}
U.~Ravat and U.~V. Shanbhag, ``On the characterization of solution sets of
  smooth and nonsmooth convex stochastic nash games,'' {\em SIAM Journal on
  Optimization}, vol.~21, no.~3, pp.~1168--1199, 2011.

\bibitem{facchinei2007finite}
F.~Facchinei and J.-S. Pang, {\em Finite-dimensional variational inequalities
  and complementarity problems}.
\newblock Springer Science \& Business Media, 2007.

\bibitem{sherali1984multiple}
H.~D. Sherali, ``A multiple leader {S}tackelberg model and analysis,'' {\em
  Operations Research}, vol.~32, no.~2, pp.~390--404, 1984.

\bibitem{pang05quasi}
J.-S. Pang and M.~Fukushima, ``Quasi-variational inequalities, generalized
  {N}ash equilibria, and multi-leader-follower games,'' {\em Computational
  Management Science}, vol.~2, no.~1, pp.~21--56, 2005.

\bibitem{su2007analysis}
C.-L. Su, ``Analysis on the forward market equilibrium model,'' {\em Operations
  Research Letters}, vol.~35, no.~1, pp.~74--82, 2007.

\bibitem{kulkarni15existence}
A.~A. Kulkarni and U.~V. Shanbhag, ``An existence result for hierarchical
  {S}tackelberg v/s {S}tackelberg games,'' {\em {IEEE} Transactions on
  Automatic Control}, vol.~60, no.~12, pp.~3379--3384, 2015.

\bibitem{steffensen18quadratic}
M.~Herty, S.~Steffensen, and A.~Th{\"u}nen, ``Solving quadratic
  multi-leader-follower games by smoothing the follower's best response,'' {\em
  Optimization Methods and Software}, pp.~1--28, 2020.

\bibitem{mpc1}
J.~B. Rawlings, D.~Q. Mayne, and M.~Diehl, {\em Model predictive control:
  theory, computation, and design}, vol.~2.
\newblock Nob Hill Publishing Madison, WI, 2017.

\bibitem{mpc2}
D.~A. Allan, C.~N. Bates, M.~J. Risbeck, and J.~B. Rawlings, ``On the inherent
  robustness of optimal and suboptimal nonlinear {MPC},'' {\em Systems Control
  Lett.}, vol.~106, pp.~68--78, 2017.

\bibitem{mpc3}
A.~Bemporad and M.~Morari, ``Robust model predictive control: A survey,'' in
  {\em Robustness in identification and control}, pp.~207--226, Springer, 1999.

\bibitem{mpc4}
F.~A. Cuzzola, J.~C. Geromel, and M.~Morari, ``An improved approach for
  constrained robust model predictive control,'' {\em Automatica}, vol.~38,
  no.~7, pp.~1183--1189, 2002.

\bibitem{cl1}
S.~Hojjatinia, C.~M. Lagoa, and F.~Dabbene, ``Identification of switched
  autoregressive exogenous systems from large noisy datasets,'' {\em Internat.
  J. Robust Nonlinear Control}, vol.~30, no.~15, pp.~5777--5801, 2020.

\bibitem{cl2}
M.~Chamanbaz, F.~Dabbene, and C.~M. Lagoa, ``Probabilistically robust {AC}
  optimal power flow,'' {\em IEEE Trans. Control Netw. Syst.}, vol.~6, no.~3,
  pp.~1135--1147, 2019.

\bibitem{cl3}
C.~Feng, F.~Dabbene, and C.~M. Lagoa, ``A kinship function approach to robust
  and probabilistic optimization under polynomial uncertainty,'' {\em IEEE
  Trans. Automat. Control}, vol.~56, no.~7, pp.~1509--1523, 2011.

\bibitem{ben2009robust}
A.~Ben-Tal, L.~El~Ghaoui, and A.~Nemirovski, {\em Robust optimization}.
\newblock Princeton university press, 2009.

\bibitem{smpc1}
P.~K. Mishra, S.~S. Diwale, C.~N. Jones, and D.~Chatterjee, ``Reference
  tracking stochastic model predictive control over unreliable channels and
  bounded control actions,'' {\em Automatica J. IFAC}, vol.~127, pp.~109512,
  10, 2021.

\bibitem{smpc2}
J.~Zhang and T.~Ohtsuka, ``Stochastic model predictive control using simplified
  affine disturbance feedback for chance-constrained systems,'' {\em IEEE
  Control Syst. Lett.}, vol.~5, no.~5, pp.~1633--1638, 2021.

\bibitem{smpc4}
Y.~Tan, Q.~Cao, L.~Li, T.~Hu, and M.~Su, ``A chance-constrained stochastic
  model predictive control problem with disturbance feedback,'' {\em J. Ind.
  Manag. Optim.}, vol.~17, no.~1, pp.~67--79, 2021.

\bibitem{smpc5}
L.~Hewing and M.~N. Zeilinger, ``Scenario-based probabilistic reachable sets
  for recursively feasible stochastic model predictive control,'' {\em IEEE
  Control Syst. Lett.}, vol.~4, no.~2, pp.~450--455, 2020.

\bibitem{smpc6}
A.~Gro\ss, C.~Wittwer, and M.~Diehl, ``Stochastic model predictive control of
  photovoltaic battery systems using a probabilistic forecast model,'' {\em
  Eur. J. Control}, vol.~56, pp.~254--264, 2020.

\bibitem{smpc7}
L.~Hewing, K.~P. Wabersich, and M.~N. Zeilinger, ``Recursively feasible
  stochastic model predictive control using indirect feedback,'' {\em
  Automatica J. IFAC}, vol.~119, pp.~109095, 7, 2020.

\bibitem{smpcreview}
A.~Mesbah, ``Stochastic model predictive control: An overview and perspectives
  for future research,'' {\em IEEE Control Systems Magazine}, vol.~36, no.~6,
  pp.~30--44, 2016.

\bibitem{camacho2013model}
E.~Camacho and C.~Alba, {\em Model Predictive Control}.
\newblock Advanced Textbooks in Control and Signal Processing, Springer London,
  2013.

\bibitem{bardakci2018probability}
I.~E. Bardakci, A.~Jalilzadeh, C.~Lagoa, and U.~V. Shanbhag, ``Probability
  maximization via minkowski functionals: Convex representations and tractable
  resolution,'' {\em arXiv preprint arXiv:1802.09682}, 2018.

\bibitem{rosasco2016stochastic}
L.~Rosasco, S.~Villa, and B.~C. V{\~u}, ``Stochastic forward--backward
  splitting for monotone inclusions,'' {\em Journal of Optimization Theory and
  Applications}, vol.~169, no.~2, pp.~388--406, 2016.

\bibitem{combettes2016stochastic}
P.~L. Combettes and J.-C. Pesquet, ``Stochastic approximations and
  perturbations in forward-backward splitting for monotone operators,'' {\em
  Pure and Applied Functional Analysis}, vol.~1, no.~1, pp.~13--37, 2016.

\bibitem{ruszczynski1997decomposition}
A.~Ruszczy{\'n}ski, ``Decomposition methods in stochastic programming,'' {\em
  Mathematical Programming}, vol.~79, no.~1-3, pp.~333--353, 1997.

\bibitem{scholschmidt2011convergence}
M.~Schmidt, N.~L. Roux, and F.~R. Bach, ``Convergence rates of inexact
  proximal-gradient methods for convex optimization,'' in {\em Advances in
  neural information processing systems}, pp.~1458--1466, 2011.

\bibitem{ghadimi15}
S.~Ghadimi, G.~Lan, and H.~Zhang, ``Mini-batch stochastic approximation methods
  for nonconvex stochastic composite optimization,'' {\em Mathematical
  Programming}, vol.~155, no.~1-2, Ser. A, pp.~267--305, 2016.

\bibitem{jalilzadeh2018smoothed}
A.~Jalilzadeh, U.~V. Shanbhag, J.~H. Blanchet, and P.~W. Glynn, ``Smoothed
  variable sample-size accelerated proximal methods for nonsmooth stochastic
  convex programs,'' {\em arXiv preprint arXiv:1803.00718}, 2018.

\bibitem{jofre2019variance}
A.~Jofr{\'e} and P.~Thompson, ``On variance reduction for stochastic smooth
  convex optimization with multiplicative noise,'' {\em Mathematical
  Programming}, vol.~174, no.~1-2, pp.~253--292, 2019.

\bibitem{rosasco2019convergence}
L.~Rosasco, S.~Villa, and B.~C. V{\~u}, ``Convergence of stochastic proximal
  gradient algorithm,'' {\em Applied Mathematics \& Optimization}, pp.~1--27,
  2019.

\bibitem{cui16analysis}
S.~Cui and U.~V. Shanbhag, ``On the analysis of reflected gradient and
  splitting methods for monotone stochastic variational inequality problems,''
  in {\em 55th {IEEE} Conference on Decision and Control, {CDC} 2016, Las
  Vegas, NV, USA, December 12-14, 2016}, pp.~4510--4515, {IEEE}, 2016.

\bibitem{yaji20analysis}
V.~G. Yaji and S.~Bhatnagar, ``Analysis of stochastic approximation schemes
  with set-valued maps in the absence of a stability guarantee and their
  stabilization,'' {\em {IEEE} Trans. Autom. Control.}, vol.~65, no.~3,
  pp.~1100--1115, 2020.

\bibitem{yaji20stochastic}
V.~G. Yaji and S.~Bhatnagar, ``Stochastic recursive inclusions in two
  timescales with nonadditive iterate-dependent markov noise,'' {\em Math.
  Oper. Res.}, vol.~45, no.~4, pp.~1405--1444, 2020.

\bibitem{cevher2018stochastic}
V.~Cevher, B.~C. V{\~u}, and A.~Yurtsever, ``Stochastic forward
  {D}ouglas-{R}achford splitting method for monotone inclusions,'' in {\em
  Large-Scale and Distributed Optimization}, pp.~149--179, Springer, 2018.

\bibitem{chen2012stochastic}
X.~Chen, R.~J.-B. Wets, and Y.~Zhang, ``Stochastic variational inequalities:
  residual minimization smoothing sample average approximations,'' {\em SIAM
  Journal on Optimization}, vol.~22, no.~2, pp.~649--673, 2012.

\bibitem{shapiro2008stochastic}
A.~Shapiro and H.~Xu, ``Stochastic mathematical programs with equilibrium
  constraints, modelling and sample average approximation,'' {\em
  Optimization}, vol.~57, no.~3, pp.~395--418, 2008.

\bibitem{FriedlanderSchmidt2012}
M.~P. Friedlander and M.~Schmidt, ``Hybrid deterministic-stochastic methods for
  data fitting,'' {\em SIAM Journal on Scientific Computing}, vol.~34, no.~3,
  pp.~A1380--A1405, 2012.

\bibitem{byrd12}
R.~H. Byrd, G.~M. Chin, J.~Nocedal, and Y.~Wu, ``Sample size selection in
  optimization methods for machine learning,'' {\em Mathematical Programming},
  vol.~134, no.~1, pp.~127--155, 2012.

\bibitem{pasupathy2018sampling}
R.~Pasupathy, P.~Glynn, S.~Ghosh, and F.~S. Hashemi, ``On sampling rates in
  simulation-based recursions,'' {\em SIAM Journal on Optimization}, vol.~28,
  no.~1, pp.~45--73, 2018.

\bibitem{ghadimi2016mini}
S.~Ghadimi, G.~Lan, and H.~Zhang, ``Mini-batch stochastic approximation methods
  for nonconvex stochastic composite optimization,'' {\em Mathematical
  Programming}, vol.~155, no.~1-2, pp.~267--305, 2016.

\bibitem{raghu17}
R.~Bollapragada, D.~Mudigere, J.~Nocedal, H.~M. Shi, and P.~T.~P. Tang, ``A
  progressive batching {L-BFGS} method for machine learning,'' in {\em
  Proceedings of the 35th International Conference on Machine Learning, {ICML}
  2018, Stockholmsm{\"{a}}ssan, Stockholm, Sweden, July 10-15, 2018} (J.~G. Dy
  and A.~Krause, eds.), vol.~80 of {\em Proceedings of Machine Learning
  Research}, pp.~619--628, {PMLR}, 2018.

\bibitem{jalilzadeh20variable}
A.~Jalilzadeh, A.~Nedic, U.~V. Shanbhag, and F.~Yousefian, ``A variable
  sample-size stochastic quasi-newton method for smooth and nonsmooth
  stochastic convex optimization,'' {\em Mathematics of Operations Research (to
  appear), https://arxiv.org/abs/1804.05368}, 2020.

\bibitem{NIPS2013_ac1dd209}
R.~Johnson and T.~Zhang, ``Accelerating stochastic gradient descent using
  predictive variance reduction,'' in {\em Advances in Neural Information
  Processing Systems} (C.~J.~C. Burges, L.~Bottou, M.~Welling, Z.~Ghahramani,
  and K.~Q. Weinberger, eds.), vol.~26, pp.~315--323, Curran Associates, Inc.,
  2013.

\bibitem{NIPS2014_ede7e2b6}
A.~Defazio, F.~Bach, and S.~Lacoste-Julien, ``{SAGA}: A fast incremental
  gradient method with support for non-strongly convex composite objectives,''
  in {\em Advances in Neural Information Processing Systems} (Z.~Ghahramani,
  M.~Welling, C.~Cortes, N.~Lawrence, and K.~Q. Weinberger, eds.), vol.~27,
  pp.~1646--1654, Curran Associates, Inc., 2014.

\bibitem{han1988parallel}
S.-P. Han and G.~Lou, ``A parallel algorithm for a class of convex programs,''
  {\em SIAM Journal on Control and Optimization}, vol.~26, no.~2, pp.~345--355,
  1988.

\bibitem{tseng2000modified}
P.~Tseng, ``A modified forward-backward splitting method for maximal monotone
  mappings,'' {\em SIAM Journal on Control and Optimization}, vol.~38, no.~2,
  pp.~431--446, 2000.

\bibitem{polyak1987introduction}
B.~T. Polyak, {\em Introduction to optimization}.
\newblock Optimization Software New York, 1987.

\bibitem{larsson94gap}
T.~Larsson and M.~Patriksson, ``A class of gap functions for variational
  inequalities,'' {\em Mathematical Programming}, vol.~64, no.~1, Ser. A,
  pp.~53--79, 1994.

\bibitem{borwein2016maximal}
J.~M. Borwein and J.~Dutta, ``Maximal monotone inclusions and {F}itzpatrick
  functions,'' {\em Journal of Optimization Theory and Applications}, vol.~171,
  no.~3, pp.~757--784, 2016.

\bibitem{borwein2010convex}
J.~M. Borwein and J.~D. Vanderwerff, {\em Convex functions: constructions,
  characterizations and counterexamples}, vol.~109.
\newblock Cambridge University Press Cambridge, 2010.

\bibitem{fitzpatrick1988representing}
S.~Fitzpatrick, ``Representing monotone operators by convex functions,'' in
  {\em Workshop/Miniconference on Functional Analysis and Optimization},
  pp.~59--65, Centre for Mathematics and its Applications, Mathematical
  Sciences Institute, The Australian National University, 1988.

\bibitem{nesterov2007dual}
Y.~Nesterov, ``Dual extrapolation and its applications to solving variational
  inequalities and related problems,'' {\em Mathematical Programming},
  vol.~109, no.~2, pp.~319--344, 2007.

\bibitem{iusem2017extragradient}
A.~Iusem, A.~Jofr{\'e}, R.~I. Oliveira, and P.~Thompson, ``Extragradient method
  with variance reduction for stochastic variational inequalities,'' {\em SIAM
  Journal on Optimization}, vol.~27, no.~2, pp.~686--724, 2017.

\bibitem{bot2020mini}
R.~Bot, P.~Mertikopoulos, M.~Staudigl, and P.~Vuong, ``Mini-batch
  forward-backward-forward methods for solving stochastic variational
  inequalities,'' {\em Stochastic Systems}, 2020.

\bibitem{ahmadi2016analysis}
H.~Ahmadi, {\em On the analysis of data-driven and distributed algorithms for
  convex optimization problems}.
\newblock The Pennsylvania State University, 2016.

\bibitem{jiang2019regularized}
J.~Jiang, Y.~Shi, X.~Wang, and X.~Chen, ``Regularized two-stage stochastic
  variational inequalities for cournot-nash equilibrium under uncertainty,''
  {\em arXiv preprint arXiv:1907.07317}, 2019.

\bibitem{dirkse1995path}
S.~P. Dirkse and M.~C. Ferris, ``The path solver: a nommonotone stabilization
  scheme for mixed complementarity problems,'' {\em Optimization Methods and
  Software}, vol.~5, no.~2, pp.~123--156, 1995.

\end{thebibliography}

\end{document}